\documentclass{amsart}
\usepackage{amsmath}
\usepackage{amssymb,amsthm}
\numberwithin{equation}{section}
\newtheorem{thm}{Theorem}[section]
\newtheorem{rem}[thm]{Remark}

\newtheorem{lem}[thm]{Lemma}

\newtheorem{dfn}[thm]{Definition}

\title[Convergence of SCF sequences]{Convergence of SCF sequences for the Hartree-Fock equation}
\author{Sohei Ashida}

\begin{document}
\maketitle

\begin{abstract}
The Hartree-Fock equation is a fundamental equation in many-electron problems. It is of practical importance in quantum chemistry to find solutions to the Hartree-Fock equation. The self-consistent field (SCF) method is a standard numerical calculation method to solve the Hartree-Fock equation. In this paper we prove that the sequence of the functions obtained in the SCF procedure is composed of a sequence of pairs of functions that converges after multiplication by appropriate unitary matrices, which strongly ensures the validity of the SCF method. A sufficient condition for the limit to be a solution to the Hartree-Fock equation after multiplication by a unitary matrix is given, and the convergence of the corresponding density operators is also proved. The method is based mainly on the proof of approach of the sequence to a critical set of a functional, compactness of the critical set, and the proof of the \L ojasiewicz inequality for another functional near critical points.
\end{abstract}

\section{Introduction and statement of the result}\label{firstsec}
Let us consider a molecule with $n$ nuclei and $N$ electrons, where $n, N\in\mathbb N$. A fundamental problem in quantum chemistry is the eigenvalue problem of the Hamiltonian
$$H:=-\sum_{i=1}^N\Delta_{x_i}-\sum_{i=1}^N\sum_{l=1}^n\frac{Z_l}{|x_i-R_l|}+\sum_{1\leq i<j\leq N}\frac{1}{|x_i-x_j|},$$
acting on $L^2(\mathbb R^{3N})$, where $x_i\in\mathbb R^3$ is the position of the $i$th electron, and $R_l$ and $Z_l$ are the position and the atomic number of the $l$th nucleus respectively. By the min-max principle (see e.g. \cite{RS4}) the eigenvalue problem is equivalent to the problem to find the critical values and the critical points of the quadratic form $\langle\Psi,H\Psi\rangle$, where $\Psi\in H^2(\mathbb R^{3N})$, $\lVert\Psi\rVert=1$.
The Hartree-Fock functional is obtained by restriction of the quadratic form to the set of all Slater determinants
$$\Psi:=(N!)^{-1/2}\sum_{\sigma\in S_N}(\mathrm{sgn}\, \sigma)\varphi_1(x_{\sigma(1)})\dotsm\varphi_N(x_{\sigma(N)}),$$
where $\varphi_i\in H^2(\mathbb R^3)$, $1\leq i\leq N$ and $\langle \varphi_i,\varphi_j\rangle =\delta_{ij}$. In other words, the Hartree-Fock functional is a functional defined by $\hat{\mathcal E}(\Phi):=\langle \Psi,H\Psi\rangle$ for $\Phi\in\mathcal W$, where
$$\mathcal W:=\left\{\Phi={}^t(\varphi_1,\dots,\varphi_N)\in \bigoplus_{i=1}^NH^2(\mathbb R^3): \langle \varphi_i,\varphi_j\rangle=\delta_{ij}\right\},$$
and $\Psi$ is the Slater determinant constructed from $\Phi={}^t(\varphi_1,\dots,\varphi_N)$. Here $\bigoplus_{i=1}^N H^2\newline(\mathbb R^3)$ is the Hilbert space equipped with the inner product $\sum_{i=1}^N\langle \varphi_i,(1-\Delta)^2\tilde\varphi_i\rangle$ for $\Phi={}^t(\varphi_1,\dots,\varphi_N)$ and $\tilde\Phi={}^t(\tilde\varphi_1,\dots,\tilde\varphi_N)$. The functional $\hat{\mathcal E}(\Phi)$ can be written explicitly as
\begin{align*}
\hat{\mathcal E}(\Phi)=\sum_{i=1}^N\langle \varphi_i,h\varphi_i\rangle&+\frac{1}{2}\int\int\rho(x)\frac{1}{|x-y|}\rho(y)dxdy\\
&-\frac{1}{2}\int\int\frac{1}{|x-y|}|\rho(x,y)|^2dxdy,
\end{align*}
where $h:=-\Delta+V$, $V(x):=-\sum_{l=1}^n\frac{Z_l}{|x-R_l|}$,
$$\rho(x):=\sum_{i=1}^N|\varphi_i(x)|^2,$$
is the density, and
$$\rho(x,y):=\sum_{i=1}^N\varphi_i(x)\varphi_i^*(y),$$
is the density matrix. Here and henceforth, $u^*$ denotes the complex conjugate for a function $u$. In this paper we will consider for simplicity of notation the spinless functions $\varphi_i$, although the results in this paper is trivially adapted to spin-dependent functions with only notational changes.

The critical values of the Hartree-Fock functional give approximations to the eigenvalues of $H$, and the corresponding critical points are used for further approximations. Let us recall the definition of critical values and critical points of a functional $\hat{\mathcal E}(\Phi):\mathcal W\to\mathbb R$ (see e.g. \cite[Definition 43.20]{Ze3}). Let $\mathcal C_{\Phi}$ be the set of all curves $c:(-1,1)\to\bigoplus_{i=1}^NH^2(\mathbb R^3)$ such that $c(t)\in \mathcal W$ for any $t\in(-1,1)$, $c(0)=\Phi$ and $c'(0)$ exists.
The point $\Phi\in \mathcal W$ is called a critical point of $\hat{\mathcal E}(\Phi)$, if $d\hat{\mathcal E}(c(t))/dt\vert_{t=0}=0$ for any $c\in \mathcal C_{\Phi}$. A real number $\Lambda\in \mathbb R$ is called a critical value of $\hat{\mathcal E}(\Phi)$ if there exists a critical point $\Phi'$ of $\hat{\mathcal E}(\Phi)$ such that $\Lambda=\hat{\mathcal E}(\Phi')$. By the method of Lagrange multipliers (see e.g. \cite[Proposition 43.21]{Ze3} and also \cite[Section 2]{As}) we can see that $\Phi$ is a critical point of $\hat{\mathcal E}(\Phi)$ if and only if there exists an Hermitian matrix $(\epsilon_{ij})$ such that $\Phi$ satisfies the equation
$$\mathcal F(\Phi)\varphi_i=\sum_{j=1}^N\epsilon_{ij}\varphi_j,\ 1\leq i\leq N.$$
Here $\mathcal F(\Phi)$ is an operator called Fock operator and defined by $\mathcal F(\Phi):=h+R^{\Phi}-S^{\Phi}$, where $R^{\Phi}$ is the multiplication operator by
$$R^{\Phi}(x):=\sum_{i=1}^N\int|x-y|^{-1}|\varphi_i(y)|^2dy=\sum_{i=1}^NQ_{ii}^{\Phi}(x),$$
with
$$Q_{ij}^{\Phi}(x):=\int|x-y|^{-1}\varphi_j^*(y)\varphi_i(y)dy,$$
and
$$S^{\Phi}:=\sum_{i=1}^NS_{ii}^{\Phi},$$
with
$$(S_{ij}^{\Phi}w)(x):=\left(\int|x-y|^{-1}\varphi_j^*(y)w(y)dy\right)\varphi_i(x),$$
for $w\in L^2(\mathbb R^3)$. We also define an operator $\mathcal G(\Phi)$ by $\mathcal G(\Phi):=R^{\Phi}-S^{\Phi}$. Then $\mathcal F(\Phi)$ can be written as $\mathcal F(\Phi)=h+\mathcal G(\Phi)$. The matrix $(\epsilon_{ij})$ is diagonalized by an $N\times N$ unitary matrix $A$ as $A(\epsilon_{ij})A^{-1}=\mathrm{diag}\, [\epsilon_1,\dots,\epsilon_N]$, and if we define new functions $\Phi^{\mathrm{New}}:={}^t(\varphi^{\mathrm{New}}_1,\dots,\varphi^{\mathrm{New}}_N)$ by $\varphi^{\mathrm{New}}_i=\sum_{j=1}^NA_{ij}\varphi^{\mathrm{Old}}_j$ from the old one $\Phi^{\mathrm{Old}}:={}^t(\varphi^{\mathrm{Old}}_1,\dots,\varphi^{\mathrm{Old}}_N)$, $\Phi^{\mathrm{New}}$ satisfies the equation
\begin{equation}\label{myeq1.1}
\mathcal F(\Phi)\varphi_i=\epsilon_i\varphi_i,\ 1\leq i\leq N,
\end{equation}
where $(\epsilon_1,\dots,\epsilon_N)\in\mathbb R^N$ and $\mathrm{diag}\, [\epsilon_1,\dots,\epsilon_N]$ is the diagonal matrix with the diagonal elements $\epsilon_1,\dots,\epsilon_N$. The equation \eqref{myeq1.1} is called Hartree-Fock equation. Hence a solution $\Phi'$ to the Hartree-Fock equation is a critical point of $\hat{\mathcal E}(\Phi)$, and $\Lambda\in\mathbb R$ is a critical value of $\hat{\mathcal E}(\Phi)$ if and only if there exists a solution $\Phi'$ to the Hartree-Fock equation such that $\Lambda=\hat{\mathcal E}(\Phi')$. The Hartree-Fock equation was introduced by Fock \cite{Fo} and Slater \cite{Sl} independently, after Hartree \cite{Ha} introduced the Hartree equation that ignored the antisymmetry with respect to exchange of variables. 
Hereafter, let us call the tuple $\mathbf e:=(\epsilon_1,\dots,\epsilon_N)$ an orbital energy, if there exists a tuple of eigenfunctions $\Phi={}^t(\varphi_1,\dots,\varphi_N)\in\mathcal W$ of $\mathcal F(\tilde \Phi)$ associated with $\mathbf e$ for some $\tilde \Phi\in\mathcal W$, i.e. we have
$$\mathcal F(\tilde\Phi)\varphi_i=\epsilon_i\varphi_i,\ i=1\leq i\leq N.$$
(Since we will consider sequences of the tuple in this paper, it would be more convenient to call the tuple an orbital energy than each $\epsilon_i$.)
In particular, the tuple $(\epsilon_1,\dots,\epsilon_N)$ of the numbers in the right-hand side of the Hartree-Fock equation \eqref{myeq1.1} is an orbital energy.

The Hartree-Fock equation can not be solved exactly even for small $n$ and $N$. A standard numerical calculation method to solve the equation is the self-consistent field (SCF) method. In the SCF method we set an initial function $\Phi^0={}^t(\varphi_1^0,\dots,\varphi_N^0)$ and repeat the following iterative procedure until the sequence $\{\Phi^k\}$ obtained in the procedure converges. Let $\varphi_1^{k+1},\dots \varphi_N^{k+1}$ be an orthonormal set of eigenfunctions of $\mathcal F(\Phi^k)$ associated with the $N$ smallest eigenvalues (including multiplicity) $\epsilon^{k+1}_1,\dots,\epsilon_N^{k+1}$, i.e.  they satisfy
$$\mathcal F(\Phi^k)\varphi_i^{k+1}=\epsilon_i^{k+1}\varphi_i^{k+1},\ 1\leq i\leq N.$$
(Assume here that we can choose such eigenfunctions, which is justified under the uniform well-posedness condition introduced later.) We set the next function as $\Phi^{k+1}:={}^t(\varphi_1^{k+1},\dots,\varphi_N^{k+1})$. 
Note that the choice of the eigenfunctions is not unique, because the multiplication by a complex number with the absolute value $1$ makes another eigenfunction, and if an eigenvalue is degenerated, multiplication by a unitary matrix to an orthonormal basis of the corresponding eigenspace generates another orthonormal basis. However, we suppose that particular eigenfanctions have been chosen in the SCF procedure.
Note that $\mathbf e^k:=(\epsilon^{k}_1,\dots,\epsilon_N^{k})$ is the orbital energy associated with $\Phi^k$. Let us call the sequence $\{\Phi^k\}$ SCF sequence.

For the analysis of the SCF sequence it is helpful to introduce operators called density operators. For $\Phi={}^t(\varphi_1,\dots,\varphi_N)\in \mathcal W$ we define the density operator $D_{\Phi}\in\mathcal L(L^2(\mathbb R^3))$ by
$$(D_{\Phi}w)(x):=\sum_{i=1}^N\left(\int\varphi_i^*(y)w(y)dy\right)\varphi_i(x),$$
for $w\in L^2(\mathbb R^3)$. We denote by $\mathcal P$ the set of operators
$$\mathcal P:=\{D\in\mathcal T_2: \mathrm{Ran}\, D\subset H^2(\mathbb R^3),\ D^2=D=D^*,\ \mathrm{Tr}\, (D) = N\},$$
where $D^*$ is the adjoint operator of $D$ and $\mathcal T_2$ is the set of all Hilbert-Schmidt operators in $L^2(\mathbb R^3)$ equipped with the norm $\lVert D\rVert_{2}:=(\mathrm{Tr}\, (D^*D))^{1/2}$ (see e.g. \cite{RS}). We can easily confirm that $D_{\Phi}\in \mathcal P$ for $\Phi\in \mathcal W$. 
For $D\in\mathcal P$ let us define an operator $G(D)$ by
$$(G(D)w)(x):=\mathrm{Tr}\, (|x-y|^{-1}D)w(x)-D(|x-y|^{-1}w(y)),$$
where $|x-y|^{-1}$ is a multiplication operator with respect to $y$ with a parameter $x$, and the trace is taken with respect to $y$. Then we can see that
$$\mathcal G(\Phi)=G(D_{\Phi}).$$
Moreover, we have
$$\hat{\mathcal E}(\Phi)=\hat E(D_{\Phi}),$$
with
$$\hat E(D):=\mathrm{Tr}\, (hD) +\frac{1}{2}\mathrm{Tr}\, (G(D)D).$$
If $\Phi={}^t(\varphi_1,\dots,\varphi_N)\in \mathcal W$ and $\tilde \Phi={}^t(\tilde \varphi_1,\dots,\tilde \varphi_N)\in\mathcal W$ satisfy $D_{\Phi}=D_{\tilde \Phi}$, then $\Phi$ and $\tilde\Phi$ are orthonormal bases of the same space $\mathrm{Ran}\, D_{\Phi}$. Hence there exists an $N\times N$ unitary matrix $A$ such that $A\Phi=\tilde \Phi$. Since the Slater determinant $\Psi$ of $\Phi$ is written as a determinant
$$\Psi=\begin{vmatrix}
\varphi_1(x_1) & \cdots & \varphi_1(x_N)\\
\vdots & \ddots & \vdots\\
\varphi_N(x_1) & \cdots & \varphi_N(x_N)
\end{vmatrix},$$
we have $\tilde \Psi=|A|\Psi$, where $\tilde\Psi$ is the Slater determinant of $\tilde \Phi$. Therefore, the possible difference between $\Psi$ and $\tilde\Psi$ is only a multiplication by a complex number with the absolute value $1$. Since in the approximation of eigenfunctions of $H$ we use $\Psi$ rather than $\Phi$, we can conclude that the multiplication by the unitary matrix $A$ is not important. In addition, since $\mathcal G(\Phi)$ is determined by $D_{\Phi}$ through $\mathcal G(\Phi)=G(D_{\Phi})$, $\Phi^{k+1}$ in the SCF sequence is determined only by $D_{\Phi^k}$. Consequently, the convergence of the density operators $D_{\Phi^k}$ is more fundamental than that of $\Phi^k$ in a certain sense.

Convergence of the SCF sequences is rarely studied from a mathematically rigorous standpoint. An important mathematically rigorous progress has been made by Canc\`es and Le Bris \cite{CB}. They introduced a functional $E(D,\tilde D):\mathcal P\times\mathcal P\to\mathbb R$ in \cite{CB} defined by
$$E(D,\tilde D):=\mathrm{Tr}\, (hD)+\mathrm{Tr}\, (h\tilde D)+\mathrm{Tr}\, (G(D)\tilde D),$$
which is symmetric with respect to $D$ and $\tilde D$. Let us also define a functional $ {\mathcal E}(\Phi,\tilde\Phi):\mathcal W\times\mathcal W\to\mathbb R$ by
\begin{align*}
{\mathcal E}(\Phi,\tilde\Phi)&:=\sum_{i=1}^N\langle\varphi_i,h\varphi_i\rangle+\sum_{i=1}^N\langle\tilde\varphi_i,h\tilde\varphi_i\rangle+\sum_{i=1}^N\langle\tilde \varphi_i,\mathcal G(\Phi)\tilde\varphi_i\rangle\\
&=\sum_{i=1}^N\langle\varphi_i,h\varphi_i\rangle+\sum_{i=1}^N\langle\tilde \varphi_i,\mathcal F(\Phi)\tilde\varphi_i\rangle\\
&=\sum_{i=1}^N\langle\varphi_i,\mathcal F(\tilde\Phi)\varphi_i\rangle+\sum_{i=1}^N\langle\tilde \varphi_i,h\tilde\varphi_i\rangle,
\end{align*}
which is symmetric with respect to $\Phi={}^t(\varphi_1,\dots,\varphi_N)$ and $\tilde\Phi={}^t(\tilde\varphi_1,\dots,\tilde\varphi_N)$. Then we have
$${\mathcal E}(\Phi,\tilde\Phi)=E(D_{\Phi},D_{\tilde\Phi}).$$
Note also that $\Phi^{k+1}={}^t(\varphi_1^{k+1},\dots,\varphi_N^{k+1})$ is the minimizer of ${\mathcal E}(\Phi^k,\Phi)$ with respect to $\Phi\in\mathcal W$. The result in \cite{CB} is that there exists a convergent subsequence $\{(D_{\Phi^{k_j}},D_{\Phi^{k_j+1}})\}$ of the sequence $\{(D_{\Phi^{k}},D_{\Phi^{k+1}})\}$ of pairs of the density operators constructed from the SCF sequence. In their analysis the fact that $E(D_{\Phi^k},D_{\Phi^{k+1}})$ is decreasing with respect to $k$ plays an important role. Their analysis is also based on the condition called uniform well-posedness. We say that a SCF sequence $\{\Phi^k\}$ is uniformly well posed, if the following condition (UWP) is fulfilled.
\begin{itemize}
\item[(UWP)] $\mathcal F(\Phi^k)$ has at least $N$ isolated eigenvalues (including multiplicity) below $\inf\sigma_{ess}(\mathcal F(\Phi^k))$ for any $k$, and there exists a constant $\gamma>0$ such that the distance between the set of the $N$ smallest eigenvalues (including multiplicity) of $\mathcal F(\Phi^k)$ and the rest of the spectrum of $\mathcal F(\Phi^k)$ is larger than or equal to $\gamma$ for any $k$, where $\sigma_{ess}(B)$ is the essential spectrum of $B$.
\end{itemize}
Note that $\sigma_{ess}(\mathcal F(\Phi^k))=[0,\infty)$ (cf. the proof of Lemma \ref{gammabound}).

Although the result in \cite{CB} was the first mathematically rigorous important step in the study of the convergence of the SCF method, existence of a convergent subsequence is essentially rather different from convergence of the sequence itself. The sequence itself may contain subsequences that even tend to infinity. In particular, in finite-dimensional spaces the existence of a convergent subsequence is an assertion strictly weaker than the assertion that the sequence is merely bounded. Another important mathematically rigorous progress has been achieved by Levitt \cite{Le} under the Galerkin discretization, i.e. finite-dimensional approximation. It is proved in \cite{Le} that $\{(D_{\Phi^{k}},D_{\Phi^{k+1}})\}$ itself converges under the finite-dimensional approximation. In \cite{Le}, in addition to the uniform well-posedness an inequality called \L ojasiewicz inequality plays a crucial role. The \L ojasiewicz inequality is the result as follows. Let $m\in\mathbb N$ and $f(x):\mathbb R^m\to\mathbb R$ be an analytic function. Then for each $x_0\in\mathbb R^m$ there exists a neighborhood $U$ of $x_0$ and two constants $\kappa>0$ and $\theta\in(0,1/2]$ such that when $x\in U$,
$$|f(x)-f(x_0)|^{1-\theta}\leq \kappa\lVert \nabla f(x)\rVert.$$

In fact, the proof of the convergence of SCF sequences with finite-dimensional approximation does not ensure the validity of the SCF method, because in general, in numerical calculations of differential equations by discretization there could exist "false solutions" that are not approximations of any true solutions to the equations. They can arise from the finite-dimensional approximations. For example, it has been theoretically proved that the equation
$$-\Delta u=u^2,$$
in $\Omega:=(0,a)\times(0,1/a),\ a>0$ with the boundary condition $u=0$ on $\partial \Omega$ does not have any  solution asymmetric with respect to the center of the region $\Omega$ in the direction of the first or second variable (cf. \cite{GNN}).  However, in a numerical calculation an asymmetric "false solution" to the equation was observed (cf. \cite{BPM}). If a sequence in the framework of the finite-dimensional approximation converges to such "false solution", the convergence is not desirable at all.

In this paper we prove that $\{(D_{\Phi^{k}},D_{\Phi^{k+1}})\}$ itself converges without any finite-dimensional approximation. Moreover, we prove that for appropriate $N\times N$ unitary matrices $A_k$, $k\geq 0$ the sequence $\{(A_{2k}\Phi^{2k},A_{2k+1}\Phi^{2k+1})\}$ converges with respect to the Sobolev norm in the direct sum of $H^2(\mathbb R^3)$. We also give a condition that a tuple of the eigenfunctions of $\mathcal F(\Xi^{\infty})$ is a solution to the Hartree-Fock equation, where
$$\Xi^{\infty}:=\lim_{k\to\infty}A_{2k}\Phi^{2k}.$$
Unlike the result above, the present result means the convergence to true solutions (at least to critical points of $\mathcal E(\Phi,\tilde\Phi)$), and this would also be the first step for the study to avoid "false solutions" as above (if they exist) in finite-dimensional approximations. The followings are the main results.

\begin{thm}\label{main}
Let $\{\Phi^k\}$ be a uniformly well posed SCF sequence such that the initial function $\Phi^0={}^t(\varphi_1^0,\dots,\varphi_n^0)$ satisfies
$$\lVert \langle x\rangle\varphi_i^0(x)\rVert\leq C_0,\ 1\leq i\leq N,$$
for some $C_0>0$.
Then there exist $\Xi^{\infty},\ \tilde\Xi^{\infty}\in \mathcal W$ such that 
\begin{align*}
&\lim_{k\to\infty}\lVert D_{\Phi^{2k}}-D_{\Xi^{\infty}}\rVert_{2}=0,\\
&\lim_{k\to\infty}\lVert D_{\Phi^{2k+1}}-D_{\tilde\Xi^{\infty}}\rVert_{2}=0.
\end{align*}
Moreover, there exists a sequence $\{A_k\}$ of $N\times N$ unitary matrices such that
\begin{align*}
&\lim_{k\to\infty}\lVert A_{2k}\Phi^{2k}-\Xi^{\infty}\rVert_{\bigoplus_{i=1}^NH^2(\mathbb R^{3})}=0,\\
&\lim_{k\to\infty}\lVert A_{2k+1}\Phi^{2k+1}-\tilde\Xi^{\infty}\rVert_{\bigoplus_{i=1}^NH^2(\mathbb R^{3})}=0.
\end{align*}
\end{thm}

\begin{rem}
In practical calculations for usual molecules, a function fulfilling the decay condition is always chosen as the initial function $\Phi^0$. Thus only the uniform well-posedness is a substantial assumption.
\end{rem}

Since the goal of the SCF method is to find a solution to the Hartree-Fock equation, and $\Phi^{k+1}$ is a tuple of eigenfunctions of $\mathcal F(\Phi^k)$, we are interested in whether a tuple $\hat\Phi^{\infty}={}^t(\hat\varphi_1^{\infty},\dots,\hat\varphi_N^{\infty})$ of eigenfunctions of $\mathcal F(\Xi^{\infty})$ corresponding to the $N$ smallest eigenvalues is a solution to the Hartree-Fock equation. (Although the choice of $\hat\Phi^{\infty}$ is not unique, assume that a particular one has been chosen.) The following theorem is concerned with this problem.

\begin{thm}\label{maincor}
Suppose the same assumption as in Theorem \ref{main}, and let $\Xi^{\infty}$ and $\tilde\Xi^{\infty}$ be as in Theorem \ref{main}. Let $\gamma>0$ be the gap in the uniform well-posedness. Then:

(1) The distance between the set of the $N$ smallest eigenvalues of $\mathcal F(\Xi^{\infty})$ (resp., $\mathcal F(\tilde\Xi^{\infty})$) and the rest of the spectrum of $\mathcal F(\Xi^{\infty})$ (resp., $\mathcal F(\tilde\Xi^{\infty})$) is larger than or equal to $\gamma$.  Thus $\hat\Phi^{\infty}$ as above is well defined.

(2) There exists an $N\times N$ unitary matrix $A_{\infty}$ such that $\tilde\Xi^{\infty}=A_{\infty}\hat\Phi^{\infty}$. Moreover, if we denote by $(\hat\epsilon^{\infty}_1,\dots,\hat\epsilon_N^{\infty})$ the eigenvalues of $\mathcal F(\Xi^{\infty})$ associated with $\hat\Phi^{\infty}={}^t(\hat\varphi_1^{\infty},\dots,\hat\varphi_N^{\infty})$, we have $\hat\epsilon_i^{\infty}=\lim_{k\to\infty}\epsilon^{2k+1}_i$, $1\leq i\leq N$.

(3) If there exists an $N\times N$ unitary matrix $\Theta$ such that $\Xi^{\infty}=\Theta\tilde\Xi^{\infty}$, then $\hat\Phi^{\infty}$ is a solution to the Hartree-Fock equation associated with the orbital energy $(\hat\epsilon^{\infty}_1,\dots,\hat\epsilon_N^{\infty})$.

(4) If $\hat\Phi^{\infty}$ forms an orthonormal basis of the direct sum of the eigenspaces of the $N$ smallest eigenvalues of $\mathcal F(\hat\Phi^{\infty})$, then there exists an $N\times N$ unitary matrix $\Theta$ such that $\Xi^{\infty}=\Theta\tilde\Xi^{\infty}$.

(5) If $D_{\hat\Phi^{\infty}}=D_{\Xi^{\infty}}$, then there exists an $N\times N$ unitary matrix $\Theta$ such that $\Xi^{\infty}=\Theta\tilde\Xi^{\infty}$.
\end{thm}

\begin{rem}
\begin{itemize}
\item[(a)] The condition $D_{\Phi}=D_{\tilde\Phi}$ is equivalent to that there exists a unitary matrix $\hat A$ such that $\Phi=\hat A\tilde\Phi$ (cf. proof of Theorem \ref{maincor} (5)). In particular, if $D_{\Xi^{\infty}}=D_{\tilde\Xi^{\infty}}$, then by Theorem \ref{maincor} (3) we can see that $\hat\Phi^{\infty}$ is a solution to the Hartree-Fock equation.
\item[(b)] There exists a case in which the SCF sequence actually fails to converge and it oscillates between two states. In \cite[Example 9]{CB} such a case is given within the Restricted Hartree-Fock (RHF) method in which the functions are spin-dependent and we impose the restriction on the tuple of functions that it consists of the same spatial functions with spin up and spin down.
\end{itemize}
\end{rem}

In the proof of Theorem \ref{main} the uniform well-posedness is used in order to obtain a bound of the difference between $D_{\Phi^k}$ and $D_{\Phi^{k+2}}$ by the difference between ${\mathcal E}(\Phi^k,\Phi^{k+1})$ and ${\mathcal E}(\Phi^{k+1},\Phi^{k+2})$ (cf. Lemma \ref{wpbound}). It also yields an upper bound of the orbital energies (cf. Lemma \ref{gammabound}) which is needed for a uniform decay estimate of $\varphi_i^k$ (cf. Lemma \ref{expbound}). The uniform decay is in turn used to prove that the SCF sequence approaches a critical set of $\hat{\mathcal E}(\Phi,\tilde \Phi)$.

In \cite{Le} due to the discretization, the known result of the \L ojasiewicz inequality for finite-dimensional cases was applicable. However, in the present result detailed infinite-dimensional analysis of functionals is needed for the proof of the \L ojasiewicz inequality. For example, we need to prove that the sequence $\{(\Phi^k,\Phi^{k+1})\}$ approaches to a critical set of ${\mathcal E}(\Phi,\tilde\Phi)$, that the critical set is a compact set, and that the Fr\'echet second derivatives of another auxiliary functional are Fredholm operators at points corresponding to the critical points of ${\mathcal E}(\Phi,\tilde\Phi)$. For such analysis the viewpoint of the function $\Phi$ is more suitable than that of density operators, particularly because the Fr\'echet second derivative of the functional of density operators is a mapping from an operator to another and difficult to analyze. Therefore, we have to relate the analysis with respect to the function $\Phi$ to that with respect to density operators. Since for any density operator there exists a corresponding class of the function $\Phi$ in which any two functions are related by a unitary matrix, we need to choose appropriate elements from the classes to obtain a relation between estimates of density operators and those of the functions. This is achieved by Lemma \ref{Ubound}.

The \L ojasiewicz inequality was proved by \L ojasiewicz \cite{Lo} for analytic functions in finite-dimensional cases. In \cite[Proposition 1.1]{HJK} the \L ojasiewicz inequality was proved for  a functional whose Fr\'echet second derivative is an isomorphism. However, the functional in the present result does not satisfy that condition. Instead its Fr\'echet second derivative at a critical point is decomposed into a sum of an isomorphism and a compact operator (Actually, the first differentiation is performed using a bilinear form as in \cite{FNSS, As, As2} to reduce the complexity due to the complex conjugate). This form of condition was first introduced in Fu\v c\'ik-Ne\v cas-Sou\v cek-Sou\v cek \cite{FNSS} for some functionals to prove that the corresponding critical values are isolated points in the set of all critical values. This condition for an auxiliary functional related to the Hartree-Fock functional was proved by Ashida \cite{As} and used to show that the number of critical values less than a constant smaller than the first energy threshold is finite. It was also a main ingredient of the proof of the fact that the set of all critical points of the Hartree-Fock functional corresponding to a critical value less than the threshold is a union of a finite number of compact connected real-analytic spaces by Ashida \cite{As2}. The way to use the \L ojasiewicz inequality in this paper is following that in \cite{Le} which was introduced by Salomon \cite{Sa} for the study of convergence of a scheme for time-discretized quantum control.

Finally let us mention the existence of solutions to the Hartree-Fock equation and the distribution of the critical values. Existence of a solution to the Hartree-Fock equation that minimizes the Hartree-Fock functional was proved by Lieb and Simon \cite{LS} under the assumption $N<\sum_{l=1}^nZ_l+1$. It was shown by Lions \cite{Li} that if $N\leq\sum_{l=1}^nZ_l$, there exists a sequence of solutions to the Hartree-Fock equation such that the corresponding critical values converge to $0$. Lewin \cite{Lew} proved that under the same assumption there exists a sequence of solutions to the Hartree-Fock equation associated with critical values converging to the first energy threshold $J(N-1)$ which is the infimum of the Hartree-Fock functional of $N-1$ electrons. For any $N$, Ashida \cite{As} showed that the set of all critical values of the Hartree-Fock functional less than $J(N-1)-\epsilon$ is finite for any $\epsilon>0$.

This paper is organized as follows. In Section \ref{thirdsec} we prove that the SCF sequences approach subsets of all critical points of $\mathcal E(\Phi,\tilde\Phi)$. The compactness of the critical sets is shown in Section \ref{fourthsec}. In Section \ref{fifthsec} an auxiliary functional is introduced and we prove that the Fr\'echet second derivative of the functional is decomposed into a sum of an isomorphism and a compact operator, if the orbital energies are tuples of negative numbers. In Section \ref{sixthsec} we show the \L ojasiewicz inequality for functionals near points at which such a decomposition is given. Finally the main theorems are proved in Section \ref{seventhsec}.

\section{approach of SCF sequences to critical sets}\label{thirdsec}
Let $\{\Phi^k\}$ be a uniformly well posed SCF sequence. Since $\Phi^{k+1}$ minimizes the functional $\Phi\mapsto{\mathcal E}(\Phi^k,\Phi)$ and ${\mathcal E}(\Phi,\tilde\Phi)$ is symmetric, we have
$$\mathcal E(\Phi^{k},\Phi^{k+1})\leq \mathcal E(\Phi^k,\Phi^{k-1})=\mathcal E(\Phi^{k-1},\Phi^k),$$
so that ${\mathcal E}(\Phi^k,\Phi^{k+1})$ is decreasing (cf. \cite{CB}). Here we note that $\mathcal G(\Phi)=R^{\Phi}-S^{\Phi}$ is a positive operator for any $\Phi={}^t(\varphi_1,\dots,\varphi_N)\in \bigoplus_{i=1}^NH^2(\mathbb R^3)$, which follows from
$$\langle w, (Q_{ii}^{\Phi}-S^{\Phi}_{ii})w\rangle=\int|x-y|^{-1}|\hat\Psi_i|^2dxdy\geq0,$$
where $\hat\Psi_i:=2^{-1/2}(w(x)\varphi_i(y)-\varphi_i(x)w(y))$. Hence, we have ${\mathcal E}(\Phi^k,\Phi^{k+1})\geq2\inf\sigma(h)$. Therefore, the limit $\mu:=\lim_{k\to\infty}{\mathcal E}(\Phi^k,\Phi^{k+1})\geq2\inf\sigma(h)$ exists.
Let $\Gamma_{\gamma,\mu}$ be the set of all solutions $(\Phi,\tilde \Phi)\in\mathcal W\times\mathcal W$ of
\begin{equation*}
\begin{split}
\mathcal F(\tilde \Phi)\varphi_i&=\epsilon_i\varphi_i\\
\mathcal F(\Phi)\tilde\varphi_i&=\tilde\epsilon_i\tilde\varphi_i
\end{split}
\qquad 1\leq i\leq N,
\end{equation*}
fulfilling ${\mathcal E}(\Phi,\tilde\Phi)=\mu$ and associated with orbital energies $\mathbf e=(\epsilon_1,\dots,\epsilon_N)\in \mathbb R^N$ and $\tilde{\mathbf e}=(\tilde\epsilon_1,\dots,\tilde\epsilon_N)\in \mathbb R^N$ satisfying $\epsilon_i,\tilde\epsilon_i\leq -\gamma$, $1\leq i\leq N$, where $\gamma$ is the gap in the uniform well-posedness, and $\Phi={}^t(\varphi_1,\dots,\varphi_N)$ and $\tilde \Phi={}^t(\tilde \varphi_1,\dots,\tilde \varphi_N)$. The set $\Gamma_{\gamma,\mu}$ is a subset of the set of all critical points of $\mathcal E(\Phi,\tilde\Phi):\mathcal W\times\mathcal W\to\mathbb R$. Let us call such a set critical set. Let
\begin{align*}
&d((\Phi^k,\Phi^{k+1}),\Gamma_{\gamma,\mu})\\
&\quad:=\inf_{(\Phi,\tilde \Phi)\in \Gamma_{\gamma,\mu}}\left(\lVert\Phi^k-\Phi\rVert_{\bigoplus_{i=1}^NH^2(\mathbb R^3)}+\lVert\Phi^{k+1}-\tilde\Phi\rVert_{\bigoplus_{i=1}^NH^2(\mathbb R^3)}\right),
\end{align*}
be the distance between $(\Phi^k,\Phi^{k+1})$ and $\Gamma_{\gamma,\mu}$ in $(\bigoplus_{i=1}^NH^2(\mathbb R^3))\bigoplus(\bigoplus_{i=1}^NH^2(\mathbb R^3))$. In this section our goal is to prove the following lemma. 

\begin{lem}\label{app}
Let $\{\Phi^k\}$ be a uniformly well posed SCF sequence such that the initial function $\Phi^0={}^t(\varphi_1^0,\dots,\varphi_n^0)$ satisfies
$$\lVert \langle x\rangle\varphi_i^0(x)\rVert\leq C_0,\ 1\leq i\leq N,$$
for some $C_0>0$.
Then we have $\lim_{k\to\infty}d((\Phi^k,\Phi^{k+1}),\Gamma_{\gamma,\mu})=0$, where $\mu:=\lim_{k\to\infty}{\mathcal E}(\Phi^k,\Phi^{k+1})$ and $\gamma$ is the gap in the uniform well-posedness.
\end{lem}

For the proof of Lemma \ref{app} we prepare several lemmas.  First under the uniform well-posedness we have the following estimate (cf. \cite{CB}).

\begin{lem}\label{wpbound}
Assume that $\{\Phi^k\}$ is a uniformly well posed SCF sequence with the gap $\gamma>0$. Then we have
$${\mathcal E}(\Phi^k,\Phi^{k+1})-{\mathcal E}(\Phi^{k+1},\Phi^{k+2})\geq2^{-1}\gamma\lVert D_{\Phi^{k+2}}-D_{\Phi^k}\rVert_{2}^2,$$
for any $k\geq 0$.
\end{lem}

\begin{proof}
First by the uniform well-posedness we have
\begin{equation}\label{myeq3.0}
\begin{split}
&{\mathcal E}(\Phi^k,\Phi^{k+1})-{\mathcal E}(\Phi^{k+1},\Phi^{k+2})\\
&\quad=\sum_{i=1}^N(\langle\varphi_i^{k},\mathcal F(\Phi^{k+1})\varphi_i^k\rangle-\langle\varphi_i^{k+2},\mathcal F(\Phi^{k+1})\varphi_i^{k+2}\rangle)\\
&\quad\geq\sum_{i=1}^N\big\{(\epsilon_N^{k+2}+\gamma)\lVert (1-E_{k+1}(\epsilon_N^{k+2}+\gamma/2))\varphi_i^k\rVert^2\\
&\qquad+\langle \varphi_i^k, \mathcal F(\Phi^{k+1})E_{k+1}(\epsilon_N^{k+2}+\gamma/2)\varphi_i^k\rangle-\epsilon_i^{k+2}\big\},
\end{split}
\end{equation}
where  $E_{k+1}(\lambda)$ is the resolution of identity of $\mathcal F(\Phi^{k+1})$, and we used that $\varphi_i^{k+2}$ is an eigenfunction of $\mathcal F(\Phi^{k+1})$ associated with $i$th eigenvalue $\epsilon_i^{k+2}$ in ascending order. Noting that the orthogonal projection of $w\in L^2(\mathbb R^3)$ onto the eigenspace corresponding to the $j$th eigenvalue of $\mathcal F(\Phi^{k+1})$ is given by $\langle \varphi_j^{k+2},w\rangle\varphi_j^{k+2}$ and $\lVert\varphi_i^k\rVert=1$, we can calculate as
$$\lVert(1-E_{k+1}(\epsilon_N^{k+2}+\gamma/2))\varphi_i^k\rVert^2=1-\sum_{j=1}^N|\langle \varphi_j^{k+2},\varphi_i^k\rangle|^2,$$
and
$$\langle \varphi_i^k, \mathcal F(\Phi^{k+1})E_{k+1}(\epsilon_N^{k+2}+\gamma/2)\varphi_i^k\rangle=\sum_{j=1}^N\epsilon_j^{k+2}|\langle\varphi_j^{k+2},\varphi_i^k\rangle|^2.$$
Thus the right-hand side of \eqref{myeq3.0} is bounded from below by
\begin{align*}
&(\epsilon_N^{k+2}+\gamma)\sum_{i=1}^N\{(1-\sum_{j=1}^N|\langle\varphi_j^{k+2},\varphi_i^k\rangle|^2)\}-\sum_{i=1}^N\{\epsilon_i^{k+2}(1-\sum_{j=1}^N|\langle\varphi_i^{k+2},\varphi_j^k\rangle|^2)\}\\
&\quad\geq\gamma(N-\sum_{i,j=1}^N|\langle\varphi_j^{k+2},\varphi_i^k\rangle|^2),
\end{align*}
where we used $\epsilon_N^{k+2}\geq\epsilon_i^{k+2}$, $1\leq i\leq N$. Therefore, we have
$${\mathcal E}(\Phi^k,\Phi^{k+1})-{\mathcal E}(\Phi^{k+1},\Phi^{k+2})\geq\gamma(N-\sum_{i,j=1}^N|\langle\varphi_j^{k+2},\varphi_i^k\rangle|^2).$$
Hence by a direct calculation we obtain
\begin{align*}
&\lVert D_{\Phi^{k+2}}-D_{\Phi^k}\rVert_{2}^2\\
&\quad=\mathrm{Tr}\, (D_{\Phi^{k+2}}^*D_{\Phi^{k+2}})-\mathrm{Tr}\, (D_{\Phi^{k+2}}^*D_{\Phi^{k}})-\mathrm{Tr}\, (D_{\Phi^{k}}^*D_{\Phi^{k+2}})+\mathrm{Tr}\, (D_{\Phi^{k}}^*D_{\Phi^{k}})\\
&\quad=2N-2\sum_{i,j=1}^N|\langle\varphi_j^{k+2},\varphi_i^k\rangle|^2\\
&\quad \leq2\gamma^{-1}({\mathcal E}(\Phi^k,\Phi^{k+1})-{\mathcal E}(\Phi^{k+1},\Phi^{k+2})),
\end{align*}
which completes the proof.
\end{proof}

A bound for the function $\Phi$ is obtained when Lemma \ref{wpbound} is combined with the following lemma. Let $\bigoplus_{i=1}^NL^2(\mathbb R^3)$ be the Hilbert space equipped with the inner product $\sum_{i=1}^N\langle \varphi_i,\tilde\varphi_i\rangle$ for $\Phi={}^t(\varphi_1,\dots,\varphi_N)$ and $\tilde\Phi={}^t(\tilde\varphi_1,\dots,\tilde\varphi_N)$.
\begin{lem}\label{Ubound}
For any $\Phi={}^t(\varphi_1,\dots,\varphi_N)\in\mathcal W$ and $\tilde\Phi={}^t(\tilde\varphi_1,\dots,\tilde\varphi_N)\in\mathcal W$ there exist $N\times N$ unitary matrices $A$ and $\tilde A$ such that
$$\lVert D_{\Phi}-D_{\tilde\Phi}\rVert_{2}\geq\lVert A\Phi-\tilde A\tilde \Phi\rVert_{\bigoplus_{i=1}^NL^2(\mathbb R^3)}.$$
\end{lem}

\begin{proof}
Let $\hat B$ be the matrix whose components are given by
$$\hat B_{ij}=\langle\varphi_i,\tilde\varphi_j\rangle.$$
By the singular value decomposition (see e.g. \cite[Theorem 1.6.3]{Ch}) there exist $N\times N$ unitary matrices $A$ and $\tilde A$ such that $\bar A\hat B({}^t\tilde A)=\mathrm{diag}\, [\lambda_1,\dotsm\lambda_N]$, where $\bar A$ is the complex conjugate of $A$ and $\lambda_1,\dots,\lambda_N$ are nonnegative real numbers that are singular values of $\hat B$. Besides, since it is easily seen that $\sup_{\mathbf c\in\mathbb C^N, |\mathbf c|=1}|\hat B\mathbf c|\leq 1$, we have $\lambda_1,\dots,\lambda_N\leq 1$. Thus setting $\Xi={}^t(\xi_1,\dots,\xi_N):=A\Phi$ and $\tilde\Xi={}^t(\tilde\xi_1,\dots,\tilde\xi_N):=\tilde A\tilde \Phi$ we obtain $\langle\xi_i,\tilde\xi_j\rangle=\delta_{ij}\lambda_i$. Moreover, we can easily see that $D_{A\Phi}=D_{\Phi}$. Hence we have
\begin{align*}
\lVert D_{\Phi}-D_{\tilde\Phi}\rVert_{2}^2&=\lVert D_{A\Phi}-D_{\tilde A\tilde\Phi}\rVert_{2}^2=2(N-\sum_{i,j=1}^N|\langle\xi_i,\tilde\xi_j\rangle|^2)\\
&=2(N-\sum_{i=1}^N\lambda_i^2)\geq2(N-\sum_{i=1}^N\lambda_i)\\
&=2(N-\sum_{i=1}^N\langle\xi_i,\tilde\xi_i\rangle)=\lVert\Xi-\tilde\Xi\rVert^2_{\bigoplus_{i=1}^NL^2(\mathbb R^3)}\\
&=\lVert A\Phi-\tilde A\tilde \Phi\rVert_{\bigoplus_{i=1}^NL^2(\mathbb R^3)}^2,
\end{align*}
which completes the proof.
\end{proof}

For the proof of the approach of the SCF sequence to $\Gamma_{\gamma,\mu}$ we need a uniform decay estimate for the functions in the sequence. The following bound on the orbital energies is necessary for the decay estimate.

\begin{lem}\label{gammabound}
Let $\{\Phi^k\}$ be a uniformly well posed SCF sequence with the gap $\gamma>0$. Then $\epsilon_i^k\leq-\gamma$, $1\leq i\leq N$ for any $k\geq 1$.
\end{lem}

\begin{proof}
Since $\epsilon_N^k=\max\{\epsilon_1^k,\dots,\epsilon_N^k\}$, we have only to prove $\epsilon_N^k\leq -\gamma$.
If we prove $\sigma_{ess}(\mathcal F(\Phi^{k-1}))=[0,\infty)$, by the uniform well-posedness we obviously have $\epsilon_N^k\leq \inf\sigma_{ess}(\mathcal F(\Phi^{k-1}))-\gamma=-\gamma$, and the proof is completed. By $\sigma_{ess}(h)=[0,\infty)$ and the Weyl's essential spectrum theorem (see e.g. \cite{RS}) we only need to prove that $\mathcal G(\Phi^{k-1})$ is $h$-compact. Since $R^{\Phi^{k-1}}(x)$ is a bounded function decaying as $|x|\to\infty$, $R^{\Phi^{k-1}}$ is $\Delta$-compact, and thus $h$-compact. Because $S^{\Phi^{k-1}}$ is an integral operator of the Hilbert-Schmidt type, it is a compact operator. Consequently, $\mathcal G(\Phi^{k-1})$ is $h$-compact, which completes the proof.
\end{proof}

Let us define $\langle x\rangle:=\sqrt{1+|x|^2}$. We denote the $L^2(\mathbb R^3)$ norm of $w\in L^2(\mathbb R^3)$ by $\lVert w\rVert$. Recall that since $\varphi_i^{k+1}$ is an eigenfunction of $\mathcal F(\Phi^k)$ associated with the eigenvalue $\epsilon_i^{k+1}$, we have
\begin{equation}\label{myeq3.2}
\mathcal F(\Phi^k)\varphi_i^{k+1}=\epsilon_i^{k+1}\varphi_i^{k+1}.
\end{equation}
The following lemma gives a uniform $H^1$ bound for the sequence.

\begin{lem}\label{Hbound}
For any $\nu>0$ there exists a constant $\tilde C_{\nu}$ such that any solution $\Phi={}^t(\varphi_1,\dots,\varphi_N)\in\mathcal W$ of
\begin{equation}\label{myeq3.0.0}
\mathcal F(\tilde\Phi)\varphi_i=\epsilon_i\varphi_i,\ 1\leq i\leq N,
\end{equation}
for some $\tilde \Phi\in\bigoplus_{i=1}^NH^2(\mathbb R^3)$ and $(\epsilon_1,\dots,\epsilon_N)\in\mathbb R^N$ with $|\epsilon_i|\leq \nu,\ 1\leq i\leq N$
satisfies $\lVert\nabla\varphi_i\rVert<\tilde C_{\nu},\ 1\leq i\leq N$.
\end{lem}

\begin{rem}\label{Hboundrem}
Assume that $\tilde\Phi\in\mathcal W$ and that $\Phi$ is a solution of \eqref{myeq3.0.0} with the orbital energy $\mathbf e=(\epsilon_1,\dots,\epsilon_N)$ satisfying $\epsilon_i\leq 0,\ 1\leq i\leq N$. Then by $\mathcal G(\tilde\Phi)\geq0$ we have
$$\epsilon_i=\langle\varphi_i,\mathcal F(\tilde\Phi)\varphi_i\rangle\geq\langle\varphi_i, h\varphi_i\rangle\geq\inf\sigma(h),$$
so that Lemma \ref{Hbound} yields $\lVert \nabla\varphi_i\rVert<\tilde C_b$, where $b:=|\inf\sigma(h)|$.
\end{rem}

\begin{proof}
By the Hardy inequality we can estimate the Coulomb potential as
$$\int\frac{1}{|x|}|w(x)|^2dx\leq\left\lVert \frac{1}{|x|}w(x)\right\rVert\lVert w\rVert\leq 2\lVert\nabla w\rVert\lVert w\rVert\leq\delta\lVert\nabla w\rVert^2+\delta^{-1}\lVert w\rVert,$$
for any $w\in H^1(\mathbb R^3)$ and $\delta>0$. Since the center of the Coulomb potential is irrelevant to the Hardy inequality, the potential $V$ in $h$ is estimated as
$$|\langle w, Vw\rangle| \leq\sum_lZ_l(\delta\lVert\nabla w\rVert^2+\delta^{-1}\lVert w\rVert^2).$$
Thus we obtain
\begin{align*}
\lVert\nabla w\rVert^2&=\langle w, (-\Delta+V)w\rangle-\langle w, Vw\rangle\\
&\leq \langle w,h w\rangle+\sum_lZ_l(\delta\lVert\nabla w\rVert^2+\delta^{-1}\lVert w\rVert^2).
\end{align*}
If we choose $\delta$ small enough so that $\delta\sum_lZ_l<1$ will hold, we have
$$\lVert \nabla w\rVert^2\leq C\langle w,hw\rangle+C\delta^{-1}\sum_lZ_l\lVert w\rVert^2,$$
where $C:=(1-\delta\sum_lZ_l)^{-1}$.
Since $\mathcal F(\tilde\Phi)=h+\mathcal G(\tilde\Phi)$ and $\mathcal G(\tilde\Phi)\geq 0$, we can see that
\begin{equation}\label{myeq3.1}
\lVert \nabla w\rVert^2\leq C\langle w,\mathcal F(\tilde\Phi)w\rangle+C\delta^{-1}\sum_lZ_l\lVert w\rVert^2.
\end{equation}
Substituting $w=\varphi_i$ into \eqref{myeq3.1} and using $\mathcal F(\tilde\Phi)\varphi_i=\epsilon_i\varphi_i$,
$\lVert\varphi_i\rVert=1$ and the assumption $|\epsilon_i|<\nu$, we obtain
\begin{equation*}\label{myeq3.1.1}
\lVert \nabla \varphi_i\rVert^2\leq \tilde C_{\nu}^2,\ 1\leq i\leq N,
\end{equation*}
where $\tilde C_{\nu}:=(C\nu+C\delta^{-1}\sum_lZ_l)^{1/2}$. This completes the proof. 
\end{proof}

In the proof of Lemma \ref{app} we uniformly estimate the norm of the functions outside spreading balls which is achieved using the following uniform decay estimate.
\begin{lem}\label{expbound}
Let $\{\Phi^k\}$ be a uniformly well posed SCF sequence. Assume that there exists a constant $C_0>0$ such that
$$\lVert \langle x\rangle\varphi_i^0(x)\rVert\leq C_0,\ 1\leq i\leq N.$$
Then there exists a constant $\mathcal C>0$ such that
\begin{equation}\label{myeq3.3.1}
\lVert \langle x \rangle\varphi_i^k(x)\rVert\leq \mathcal C,\ 1\leq i\leq N,
\end{equation}
for any $k\geq 0$.
\end{lem}

\begin{proof}
Let $k$ be fixed and let us assume that there exists $C_k$ such that
\begin{equation}\label{myeq3.4}
\lVert\langle x\rangle\varphi_i^k(x)\rVert^2\leq C_k,\ 1\leq i\leq N.
\end{equation}
We shall seek $C_{k+1}$ so that \eqref{myeq3.4} will hold with $k$ replaced by $k+1$.
Let $\eta(r)\in C_0^{\infty}(\mathbb R)$ be a function such that $\eta(r)=r$ for $-1<r<1$ and $|\eta'(r)|\leq 1$. For any $m\in \mathbb N$ we set $\rho_m(x):=m\eta(\langle x\rangle/m)$.
By a direct calculation we have
\begin{align*}
\mathrm{Re}\, \langle(-\Delta\varphi_i^{k+1}),\rho_m^2\varphi_i^{k+1}\rangle&=\lVert\nabla(\rho_m\varphi_i^{k+1})\rVert^2-\lVert(\nabla \rho_m)\varphi_i^{k+1}\rVert^2.\\
\end{align*}
Thus by \eqref{myeq3.2} we obtain
\begin{equation}\label{myeq3.4.1}
\begin{split}
0&=\mathrm{Re}\, \left\langle(-\Delta+V(x)+R^{\Phi^k}(x)-\epsilon_i^{k+1})\varphi_i^{k+1}-S^{\Phi^k}\varphi_i^{k+1},\rho_m^2\varphi_i^{k+1}\right\rangle\\
&=\lVert\nabla(\rho_m\varphi_i^{k+1})\rVert^2-\lVert(\nabla \rho_m)\varphi_i^{k+1}\rVert^2\\
&\quad+\langle (V(x)+R^{\Phi^k}(x)-\epsilon_i^{k+1})\varphi_i^{k+1},\rho_m^2\varphi_i^{k+1}\rangle\\
&\quad-\mathrm{Re}\, \langle S^{\Phi^k}\varphi_i^{k+1},\rho_m^2\varphi_i^{k+1}\rangle\\
&\geq-\lVert(\nabla \rho_m)\varphi_i^{k+1}\rVert^2+\langle (V(x)+R^{\Phi^k}(x)-\epsilon_i^{k+1})\varphi_i^{k+1},\rho_m^2\varphi_i^{k+1}\rangle\\
&\quad-\mathrm{Re}\, \langle S^{\Phi^k}\varphi_i^{k+1},\rho_m^2\varphi_i^{k+1}\rangle\\
&\geq-1+\langle (V(x)+R^{\Phi^k}(x)-\epsilon_i^{k+1})\varphi_i^{k+1},\rho_m^2\varphi_i^{k+1}\rangle-\mathrm{Re}\, \langle S^{\Phi^k}\varphi_i^{k+1},\rho_m^2\varphi_i^{k+1}\rangle,
\end{split}
\end{equation}
where we used $\lVert \varphi_i^{k+1}\rVert=1$ and that $|\nabla\rho_m(z)|\leq 1$ for any $z\in\mathbb R^3$.

Here we note that
$$\lvert\rho_m(x)-\rho_m(y)\rvert =\left|\int_0^1(x-y)\cdot\nabla\rho_m(t(x-y)+y)dt\right|\leq|x-y|.$$
Thus we have
\begin{align*}
&|\langle S_{jj}^{\Phi^k}\varphi_i^{k+1},\rho_m^2\varphi_i^{k+1}\rangle-\langle S_{jj}^{\Phi^k}\rho_m\varphi_i^{k+1},\rho_m\varphi_i^{k+1}\rangle|\\
&\quad=\left|\int|x-y|^{-1}\varphi_j^k(y)\overline{\varphi_i^{k+1}}(y)\rho_m(x)(\rho_m(x)-\rho_m(y))\overline{\varphi_j^k}(x)\varphi_i^{k+1}(x)dxdy\right|\\
&\quad\leq\int\left|\varphi_j^k(y)\overline{\varphi_i^{k+1}}(y)\rho_m(x)\overline{\varphi_j^k}(x)\varphi_i^{k+1}(x)\right|dxdy\\
&\quad\leq \lVert|\rho_m|^{1/2}\varphi_j^{k}\rVert\lVert|\rho_m|^{1/2}\varphi_i^{k+1}\rVert,
\end{align*}
where $\overline u$ is the complex conjugate of $u$.
Since the factors in the right-hand side are estimated as
\begin{align*}
\lVert|\rho_m|^{1/2}\varphi_i^{k+1}\rVert&=\left(\int|\rho_m(x)||\varphi_i^{k+1}(x)|^2dx\right)^{1/2}\\
&\leq\lVert\rho_m\varphi_i^{k+1}\rVert^{1/2}\lVert\varphi_i^{k+1}\rVert^{1/2}=\lVert\rho_m\varphi_i^{k+1}\rVert^{1/2},
\end{align*}
we obtain
\begin{align*}
&|\langle S_{jj}^{\Phi^k}\varphi_i^{k+1},\rho_m^2\varphi_i^{k+1}\rangle-\langle S_{jj}^{\Phi^k}\rho_m\varphi_i^{k+1},\rho_m\varphi_i^{k+1}\rangle|\\
&\quad\leq\lVert\rho_m\varphi_j^{k}\rVert^{1/2}\lVert\rho_m\varphi_i^{k+1}\rVert^{1/2}\\
&\quad\leq (2\gamma)^{-1}N+ (2N)^{-1}\gamma\lVert\rho_m\varphi_j^{k}\rVert\lVert\rho_m\varphi_i^{k+1}\rVert\\
&\quad\leq (2\gamma)^{-1}N+(4N)^{-1}\gamma\lVert\rho_m\varphi_j^{k}\rVert^2+ (4N)^{-1}\gamma\lVert\rho_m\varphi_i^{k+1}\rVert^2\\
&\quad\leq (2\gamma)^{-1}N+(4N)^{-1}\gamma C_k+ (4N)^{-1}\gamma\lVert\rho_m\varphi_i^{k+1}\rVert^2,
\end{align*}
where $\gamma$ is the gap in the uniform well-posedness. Therefore, we have
\begin{align*}
&|\langle S^{\Phi^k}\varphi_i^{k+1},\rho_m^2\varphi_i^{k+1}\rangle-\langle S^{\Phi^k}\rho_m\varphi_i^{k+1},\rho_m\varphi_i^{k+1}\rangle|\\
&\quad\leq (2\gamma)^{-1}N^2+4^{-1}\gamma C_k+ 4^{-1}\gamma\lVert\rho_m\varphi_i^{k+1}\rVert^2.
\end{align*}
Thus by \eqref{myeq3.4.1} and $\langle w,(R^{\Phi^k}-S^{\Phi^k})w\rangle\geq0$ with $w=\rho_m\varphi_i^{k+1}$ we obtain
\begin{equation}\label{myeq3.7}
\begin{split}
0&\geq-1-(2\gamma)^{-1}N^2-4^{-1}\gamma C_k-4^{-1}\gamma\lVert\rho_m\varphi_i^{k+1}\rVert^2\\
&\quad+\langle (V(x)-\epsilon_i^{k+1})\varphi_i^{k+1},\rho_m^2\varphi_i^{k+1}\rangle\\
&\geq-1-(2\gamma)^{-1}N^2-4^{-1}\gamma C_k+\langle (V(x)+(3/4)\gamma)\varphi_i^{k+1},\rho_m^2\varphi_i^{k+1}\rangle,
\end{split}
\end{equation}
where we used $\epsilon_i^{k+1}\leq-\gamma$ of Lemma \ref{gammabound} in the second inequality.

Now let $r_0>0$ be a constant such that $|V(x)|<\frac{\gamma}{4}$ for $|x|>r_0$. Then decomposing the integral in \eqref{myeq3.7} into those on $|x|\leq r_0$ and $|x|> r_0$ we have
\begin{align*}
&2^{-1}\gamma\int_{|x|>r_0}\rho_m^2(x)|\varphi_i^{k+1}(x)|^2dx\\
&\quad\leq 1+(2\gamma)^{-1}N^2+4^{-1}\gamma C_k+\int_{|x|\leq r_0}\left|V(x)+(3/4)\gamma\right|\rho_m^2(x)|\varphi_i^{k+1}(x)|^2dx\\
&\quad\leq 1+(2\gamma)^{-1}N^2+4^{-1}\gamma C_k+(1+r_0^2)\left(2\sum_lZ_l\lVert\nabla\varphi_i^{k+1}\rVert+(3/4)\gamma\right)\\
&\quad\leq 1+(2\gamma)^{-1}N^2+4^{-1}\gamma C_k+(1+r_0^2)\left(2\sum_lZ_l\tilde C_b+(3/4)\gamma\right),
\end{align*}
where we used $|\rho_m(x)|\leq \langle x\rangle$ and the Hardy inequality in the second inequality, and $\tilde C_b$ is the constant in Remark \ref{Hboundrem}. Hence Fatou's lemma yields
\begin{align*}
\int_{|x|>r_0}\langle x\rangle^2|\varphi_i^{k+1}(x)|^2dx&=\liminf_{m\to\infty}\int_{|x|>r_0}\rho_m^2(x)|\varphi_i^{k+1}(x)|^2dx\\
&\leq2^{-1}C_k+\hat C,
\end{align*}
where $\hat C:=2\gamma^{-1}\{1+(2\gamma)^{-1}N^2+(1+r_0^2)(2\sum_lZ_l\tilde C_b+(3/4)\gamma)\}$ is independent of $k$. Therefore, noting that
$$\int_{|x|\leq r_0}\langle x\rangle^2|\varphi_i^{k+1}(x)|^2dx\leq 1+r_0^2,$$
we obtain
$$\lVert \langle x\rangle\varphi_i^{k+1}\rVert^2=\int\langle x\rangle^2|\varphi_i^{k+1}(x)|^2dx\leq 2^{-1}C_k+\hat C +1+r_0^2.$$

Thus setting $\check C:= \hat C+1+r_0^2$ we can choose $C_{k+1}=2^{-1}C_k+\check C$ in \eqref{myeq3.4} with $k$ replaced by $k+1$. Then we can easily see that
$$C_k=2^{-k}C_0+\check C\sum_{j=0}^{k-1}2^{-j}\leq C_0+2\check C,$$
for any $k\geq 1$. Therefore, we can choose $\mathcal C:=(C_0+2\check C)^{1/2}$ as the constant in \eqref{myeq3.3.1}, which completes the proof.
\end{proof}

\begin{proof}[Proof of Lemma \ref{app}]
We have only to prove that any subsequence of $\{(\Phi^k,\Phi^{k+1})\}$ contains a subsequence converging to a point in $\Gamma_{\gamma,\mu}$. Lemma \ref{app} follows from this assertion as follows. Suppose $d((\Phi^k,\Phi^{k+1}),\Gamma_{\gamma,\mu})$ does not converge to $0$ against the result of Lemma \ref{app}. Then we can choose a constant $\delta>0$ and a subsequence $\{(\Phi^{k_j},\Phi^{k_j+1})\}$ such that $d((\Phi^{k_j},\Phi^{k_j+1}),\Gamma_{\gamma,\mu})\geq\delta$ for any $j$, which contradicts the assertion above.

\noindent\textbf{Step 1.}
First we shall prove that any subsequence $\{\Phi^{k_j}\}$ of $\{\Phi^k\}$ contains a convergent subsequence in $\bigoplus_{i=1}^NL^2(\mathbb R^3)$. This can be proved in a way similar to the proof of \cite[Lemma 3.3]{As}. By Lemma \ref{Hbound} and the Rellich selection theorem for any $p\in\mathbb N$ there exists a Cauchy subsequence of $\{\Phi^{k_j}\}$ in $\bigoplus_{i=1}^NL^2(B_p)$, still denoted by $\{\Phi^{k_j}\}$, where $B_r:=\{x\in\mathbb R^3:|x|<r\}$. The Cauchy sequence satisfies
$$\lVert \Phi^{k_{j_1}}-\Phi^{k_{j_2}}\rVert_{\bigoplus_{i=1}^NL^2(B_p)}\to 0,$$
as $j_1, j_2\to\infty$.
Thus we can choose further a subsequence, still denoted by $\{\Phi^{k_j}\}$, such that
$$\lVert \Phi^{k_{j_1}}-\Phi^{k_{j_2}}\rVert_{\bigoplus_{i=1}^NL^2(B_{j_0})}<j_0^{-1},$$
where $j_0:=\min\{j_1,j_2\}$. By Lemma \ref{expbound} there exists a constant $\tilde{\mathcal C}$ such that $\lVert \langle x\rangle\Phi^{k}\rVert_{\bigoplus_{i=1}^NL^2(\mathbb R^3)}\leq \tilde{\mathcal C}$ for any $k$, where $\langle x\rangle\Phi^{k}:={}^t(\langle x\rangle\varphi_1^{k},\dots,\langle x\rangle\varphi_N^{k})$. Since $|x|\geq j$ for $x\in\mathbb R^3\setminus B_j$, we have
$$\lVert\Phi^{k_j}\rVert_{\bigoplus_{i=1}^NL^2(\mathbb R^3\setminus B_j)}\leq j^{-1}\lVert\langle x\rangle\Phi^{k_j}\rVert_{\bigoplus_{i=1}^NL^2(\mathbb R^3\setminus B_j)}\leq \tilde{\mathcal C}j^{-1}.$$
Therefore, we obtain
\begin{align*}
&\lVert \Phi^{k_{j_1}}-\Phi^{k_{j_2}}\rVert_{\bigoplus_{i=1}^NL^2(\mathbb R^3)}\\
&\quad\leq\lVert \Phi^{k_{j_1}}-\Phi^{k_{j_2}}\rVert_{\bigoplus_{i=1}^NL^2(B_{j_0})}+\lVert \Phi^{k_{j_1}}-\Phi^{k_{j_2}}\rVert_{\bigoplus_{i=1}^NL^2(\mathbb R^3\setminus B_{j_0})}\\
&\quad\leq j_0^{-1}+2\tilde{\mathcal C}j_0^{-1}.
\end{align*}
Thus $\{\Phi^{k_j}\}$ is a Cauchy sequence in $\bigoplus_{i=1}^NL^2(\mathbb R^3)$.

\noindent\textbf{Step 2.} By the same argument as above we can see that there exists a Cauchy subsequence of $\{\Phi^{k_j-1}\}$ in $\bigoplus_{i=1}^NL^2(\mathbb R^3)$. Hence we can extract a Cauchy subsequence of $\{(\Phi^{k_j-1},\Phi^{k_j})\}$ in $(\bigoplus_{i=1}^NL^2(\mathbb R^3))\oplus(\bigoplus_{i=1}^NL^2(\mathbb R^3))$, still denoted by $\{(\Phi^{k_j-1},\Phi^{k_j})\}$. Since by Lemma \ref{gammabound} and Remark \ref{Hboundrem} we have $\mathbf e^k\in[\inf\sigma(h),-\gamma]^N$, we can further extract a subsequence so that $\mathbf e^{k_j}$ will be a Cauchy sequence. Then using the equation $\mathcal F(\Phi^{k_j-1})\varphi_i^{k_j}=\epsilon_i^{k_j}$, we can see that
\begin{equation}\label{myeq3.8}
\begin{split}
&\lVert h(\varphi_i^{k_{j_1}}-\varphi_i^{k_{j_2}})\rVert\\
&\quad\leq\lVert(\epsilon_i^{k_{j_1}}-R^{\Phi^{k_{j_1}-1}}+S^{\Phi^{k_{j_1}-1}})\varphi_i^{k_{j_1}}-(\epsilon_i^{k_{j_2}}-R^{\Phi^{k_{j_2}-1}}+S^{\Phi^{k_{j_2}-1}})\varphi_i^{k_{j_2}}\rVert.
\end{split}
\end{equation}
Noting that by the Hardy inequality we have estimates as
\begin{align*}
\left|\int|x-y|^{-1}(\varphi_i^{k_{j_1}-1}-\varphi_i^{k_{j_2}-1})^*(y)\varphi_i^{k_{j_1}}(y)dy\right|&\leq2\lVert\varphi_i^{k_{j_1}-1}-\varphi_i^{k_{j_2}-1}\rVert\lVert\nabla\varphi_i^{k_{j_1}}\rVert\\
&\leq2\tilde C_b\lVert\varphi_i^{k_{j_1}-1}-\varphi_i^{k_{j_2}-1}\rVert,
\end{align*}
with the constant $\tilde C_b$ in Remark \ref{Hboundrem}, it follows from \eqref{myeq3.8} that there exists a constant $\hat C_1>0$ such that
\begin{align*}
&\lVert h(\varphi_i^{k_{j_1}}-\varphi_i^{k_{j_2}})\rVert\\
&\quad\leq \hat C_1(\lVert \Phi^{k_{j_1}}-\Phi^{k_{j_2}}\rVert_{\bigoplus_{i=1}^NL^2(\mathbb R^3)}+\lVert \Phi^{k_{j_1}-1}-\Phi^{k_{j_2}-1}\rVert_{\bigoplus_{i=1}^NL^2(\mathbb R^3)}\\
&\qquad+|\mathbf e^{k_{j_1}}-\mathbf e^{k_{j_2}}|).
\end{align*}
Because $V$ is $\Delta$-bounded with a relative bound smaller than $1$, $\Delta$ is $h$-bounded, and therefore, we can conclude that $\{\Phi^{k_{j}}\}$ is a Cauchy sequence in $\bigoplus_{i=1}^NH^2(\mathbb R^3)$.

\noindent\textbf{Step 3.} 
In the same way as above we can see that there exists a convergent subsequence of $\{\Phi^{k_j+1}\}$ in $\bigoplus_{i=1}^NH^2(\mathbb R^3)$. Besides there exists a convergent subsequence of $\{(\mathbf e^{k_j},\mathbf e^{k_j+1})\}$. Hence we can extract a Cauchy subsequence of $\{(\Phi^{k_j},\Phi^{k_j+1})\}$ in $(\bigoplus_{i=1}^NH^2(\mathbb R^3))\bigoplus(\bigoplus_{i=1}^NH^2(\mathbb R^3))$, still denoted by $\{(\Phi^{k_j},\Phi^{k_j+1})\}$, such that $\{(\mathbf e^{k_j},\mathbf e^{k_j+1})\}$ also converges. Set
$$(\Phi^{\infty},\tilde\Phi^{\infty}):=\lim_{j\to\infty}\{(\Phi^{k_j},\Phi^{k_j+1})\},$$
where $\Phi^{\infty}={}^t(\varphi^{\infty}_1,\dots,\varphi_N^{\infty})$, $\tilde\Phi^{\infty}={}^t(\tilde\varphi^{\infty}_1,\dots,\tilde\varphi_N^{\infty})$ and
$$(\mathbf e^{\infty},\tilde{\mathbf e}^{\infty}):=\lim_{j\to\infty}\{(\mathbf e^{k_j},\mathbf e^{k_j+1})\},$$
where $\mathbf e^{\infty}=(\epsilon_1^{\infty},\dots,\epsilon_N^{\infty})$, $\tilde{\mathbf e}^{\infty}=(\tilde\epsilon_1^{\infty},\dots,\tilde\epsilon_N^{\infty})$.
Taking the limits in $L^2(\mathbb R^3)$ of the both sides of
$$\mathcal F(\Phi^{k_j})\varphi_i^{k_j+1}=\epsilon_i^{k_j+1}\varphi_i^{k_j+1},$$
we obtain
\begin{equation}\label{myeq3.8.1}
\mathcal F(\Phi^{\infty})\tilde\varphi_i^{\infty}=\tilde\epsilon_i^{\infty}\tilde\varphi_i^{\infty}.
\end{equation}

In order to consider the convergence of the other equation
\begin{equation}\label{myeq3.9}
\mathcal F(\Phi^{k_j-1})\varphi_i^{k_j}=\epsilon_i^{k_j}\varphi_i^{k_j},
\end{equation}
we shall prove $\lim_{j\to\infty}\lVert\mathcal F(\Phi^{k_j+1})\varphi_i^{k_j}-\mathcal F(\Phi^{k_j-1})\varphi_i^{k_j}\rVert=\lim_{j\to\infty}\lVert\mathcal G(\Phi^{k_j+1})\varphi_i^{k_j}-\mathcal G(\Phi^{k_j-1})\varphi_i^{k_j}\rVert=0$. Recall that ${\mathcal E}(\Phi^k,\Phi^{k+1})$ converges to $\mu$. Hence by Lemmas \ref{wpbound} and \ref{Ubound} for any $\delta>0$ there exists $j_0$ such that for any $j\geq j_0$ with appropriate $N\times N$ unitary matrices $\check A_{k_j-1}^-, \check A_{k_j+1}^+$ we have
\begin{align*}
&\lVert \check A_{k_j+1}^+\Phi^{k_{j}+1}-\check A_{k_j-1}^-\Phi^{k_{j}-1}\rVert_{\bigoplus_{i=1}^NL^2(\mathbb R^3)}^2\\
&\quad\leq \lVert D_{\Phi^{k_{j}+1}}-D_{\Phi^{k_{j}-1}}\rVert_{2}^2\\
&\quad\leq 2\gamma^{-1}({\mathcal E}(\Phi^{k_{j}-1},\Phi^{k_{j}})-{\mathcal E}(\Phi^{k_{j}},\Phi^{k_{j}+1}))\leq \delta.
\end{align*}
Note also that by Remark \ref{Hboundrem} there exists a constant $\hat C_2>0$ independent of $j$ such that
\begin{align*}
\lVert \check A_{k_j+1}^+\Phi^{k_{j}+1}\rVert_{\bigoplus_{i=1}^NH^1(\mathbb R^3)}&=\lVert \Phi^{k_{j}+1}\rVert_{\bigoplus_{i=1}^NH^1(\mathbb R^3)}\leq \hat C_2,\\
\lVert \check A_{k_j-1}^-\Phi^{k_{j}-1}\rVert_{\bigoplus_{i=1}^NH^1(\mathbb R^3)}&=\lVert \Phi^{k_{j}-1}\rVert_{\bigoplus_{i=1}^NH^1(\mathbb R^3)}\leq \hat C_2.
\end{align*}
Thus we can see that there exists a constant $\hat C_3>0$ such that for $j\geq j_0$
\begin{align*}
&\lVert\mathcal G(\Phi^{k_j+1})\varphi_i^{k_j}-\mathcal G(\Phi^{k_j-1})\varphi_i^{k_j}\rVert\\
&\quad=\lVert\mathcal G(\check A_{k_j+1}^+\Phi^{k_j+1})\varphi_i^{k_j}-\mathcal G(\check A_{k_j-1}^-\Phi^{k_j-1})\varphi_i^{k_j}\rVert\\
&\quad\leq \hat C_3\lVert\check  A_{k_j+1}^+\Phi^{k_{j}+1}-\check A_{k_j-1}^-\Phi^{k_{j}-1}\rVert_{\bigoplus_{i=1}^NL^2(\mathbb R^3)}\\
&\qquad\cdot(\lVert \check A_{k_j+1}^+\Phi^{k_{j}+1}\rVert_{\bigoplus_{i=1}^NH^1(\mathbb R^3)}+\lVert\check  A_{k_j-1}^-\Phi^{k_{j}-1}\rVert_{\bigoplus_{i=1}^NH^1(\mathbb R^3)}+\lVert\varphi_i^{k_j}\rVert_{H^1(\mathbb R^3)})\\
&\quad\leq 3\hat C_2\hat C_3\lVert \check A_{k_j+1}^+\Phi^{k_{j}+1}-\check A_{k_j-1}^-\Phi^{k_{j}-1}\rVert_{\bigoplus_{i=1}^NL^2(\mathbb R^3)}\leq 3\hat C_2\hat C_3\delta^{1/2}.
\end{align*}
Since we can choose arbitrarily small $\delta$, This implies
$$\lim_{j\to\infty}\lVert\mathcal G(\Phi^{k_j+1})\varphi_i^{k_j}-\mathcal G(\Phi^{k_j-1})\varphi_i^{k_j}\rVert=0.$$
Thus we have $\lim_{j\to\infty}\mathcal F(\Phi^{k_j-1})\varphi_i^{k_j}=\lim_{j\to\infty}\mathcal F(\Phi^{k_j+1})\varphi_i^{k_j}=\mathcal F(\tilde\Phi^{\infty})\varphi_i^{\infty}$ in $L^2(\mathbb R^3)$. Hence taking the limits in the both sides of \eqref{myeq3.9} we obtain
\begin{equation}\label{myeq3.10}
\mathcal F(\tilde\Phi^{\infty})\varphi_i^{\infty}=\epsilon_i^{\infty}\varphi_i^{\infty}.
\end{equation}
The conditions ${\mathcal E}(\Phi^{\infty},\tilde\Phi^{\infty})=\mu$ and $\epsilon^{\infty}_i, \tilde\epsilon^{\infty}_i\leq-\gamma$ follow from the definition $\mu=\lim_{k\to\infty}{\mathcal E}(\Phi^k,\Phi^{k+1})$ and Lemma \ref{gammabound}, and therefore, by \eqref{myeq3.8.1} and \eqref{myeq3.10} we have $(\Phi^{\infty},\tilde\Phi^{\infty})\in \Gamma_{\gamma,\mu}$, which completes the proof.
\end{proof}

\section{Compactness of critical sets}\label{fourthsec}
Let  $\Gamma_{\gamma,\mu}$ be the set defined at the beginning of Section \ref{thirdsec}.
\begin{lem}\label{compact}
For any $\gamma>0$ and $\mu\in\mathbb R$, the set $\Gamma_{\gamma,\mu}$ is a compact subset of $(\bigoplus_{i=1}^NH^2(\mathbb R^3))\bigoplus(\bigoplus_{i=1}^NH^2(\mathbb R^3))$.
\end{lem}

The proof of this lemma is similar to that of Lemma \ref{app}. Therefore, we prepare the corresponding decay estimate.
\begin{lem}\label{weightbound2}
Let $\gamma>0$ and $\mu\in\mathbb R$. Then there exists a constant $C'_{\gamma}$ such that for any $(\Phi,\tilde\Phi)\in \Gamma_{\gamma,\mu}$ we have
$$\lVert\langle x\rangle\Phi\rVert_{\bigoplus_{i=1}^NL^2(\mathbb R^3)},\lVert\langle x\rangle\tilde\Phi\rVert_{\bigoplus_{i=1}^NL^2(\mathbb R^3)}\leq C'_{\gamma}.$$
\end{lem}

\begin{proof}
By exactly the same way as \eqref{myeq3.7} we obtain
\begin{align*}
0\geq&-1-(2\gamma)^{-1}N^2-(4N)^{-1}\gamma\sum_{j=1}^N\lVert\rho_m\tilde\varphi_j\rVert^2-4^{-1}\gamma \lVert\rho_m\varphi_i\rVert^2\\
&+\langle (V(x)-\epsilon_i)\varphi_i,\rho_m^2\varphi_i\rangle,\quad 1\leq i\leq N,
\end{align*}
and
\begin{align*}
0\geq&-1-(2\gamma)^{-1}N^2-(4N)^{-1}\gamma\sum_{j=1}^N\lVert\rho_m\varphi_i\rVert^2-4^{-1}\gamma \lVert\rho_m\tilde\varphi_i\rVert^2\\
&+\langle (V(x)-\tilde\epsilon_i)\tilde\varphi_i,\rho_m^2\tilde\varphi_i\rangle,\quad 1\leq i\leq N.
\end{align*}
Adding the both sides of the inequalities for $1\leq i\leq N$ and noting that $\epsilon_i, \tilde\epsilon_i\leq-\gamma$ we have
$$0\geq-2N-\gamma^{-1}N^3+\sum_{i=1}^N\langle (V(x)+\gamma/2)\varphi_i,\rho_m^2\varphi_i\rangle+\sum_{i=1}^N\langle (V(x)+\gamma/2)\varphi_i,\rho_m^2\varphi_i\rangle.$$
Let $r_1>0$ be a constant such that $|V(x)|\leq \gamma/4$ for $|x|>r_1$. Decomposing the integral into those on $|x|\leq r_1$ and $|x|> r_1$ we have
\begin{align*}
&4^{-1}\gamma\sum_{i=1}^N\int_{|x|>r_1}\rho_m^2(x)(|\varphi_i(x)|^2+|\tilde\varphi_i(x)|^2)dx\\
&\quad\leq2N+\gamma^{-1}N^3+\sum_{i=1}^N\int_{|x|\leq r_1}\rho_m^2(x) (|V(x)|+\gamma/2)(|\varphi_i(x)|^2+|\tilde\varphi_i(x)|^2)dx\\
&\quad\leq2N+\gamma^{-1}N^3+N(1+r_1^2)\left(4\sum_lZ_l\tilde C_b+\gamma\right),
\end{align*}
where $\tilde C_b$ is the constant in Remark \ref{Hboundrem}.
Fatou's lemma yields
\begin{align*}
&4^{-1}\gamma\sum_{i=1}^N\int_{|x|>r_1}\langle x\rangle^2(|\varphi_i(x)|^2+|\tilde\varphi_i(x)|^2)dx\\
&\quad\leq2N+\gamma^{-1}N^3+N(1+r_1^2)\left(4\sum_lZ_l\tilde C_b+\gamma\right).
\end{align*}
Noting that
$$\int_{|x|\leq r_1}\langle x\rangle^2(|\varphi_i(x)|^2+|\tilde\varphi_i(x)|^2)dx\leq2(1+r_1^2),$$
we obtain
\begin{align*}
&\sum_{i=1}^N(\lVert \langle x\rangle \varphi_i\rVert^2+\lVert \langle x\rangle \tilde\varphi_i\rVert^2)\\
&\quad\leq4\gamma^{-1}\left(2N+\gamma^{-1}N^3+N(1+r_1^2)\left(4\sum_lZ_l\tilde C_b+\gamma\right)\right)+2N(1+r_1^2).
\end{align*}
Thus if we set
\begin{align*}
C'_{\gamma}:=&\Bigg\{4\gamma^{-1}\left(2N+\gamma^{-1}N^3+N(1+r_1^2)\left(4\sum_lZ_l\tilde C_b+\gamma\right)\right)\\
&+2N(1+r_1^2)\Bigg\}^{1/2},
\end{align*}
the result follows.
\end{proof}

\begin{proof}[Proof of Lemma \ref{compact}]
Let $\{(\Phi^k,\tilde\Phi^k)\}\subset \Gamma_{\gamma,\mu}$ be an arbitrary sequence in $\Gamma_{\gamma,\mu}$.
Using Lemma \ref{weightbound2} in the same way as in the proof of Lemma \ref{app} we can see that there exists a subsequence $\{(\Phi^{k_j},\tilde\Phi^{k_j})\}$ of $\{(\Phi^k,\tilde\Phi^k)\}$ converging to a point $(\Phi^{\infty},\tilde\Phi^{\infty})$ in $(\bigoplus_{i=1}^NH^2(\mathbb R^3))\bigoplus(\bigoplus_{i=1}^NH^2(\mathbb R^3))$, and the associated orbital energies $\mathbf e^{k_j}$ and $\tilde{\mathbf e}^{k_j}$ converge to some $\mathbf e^{\infty}=(\epsilon_1^{\infty},\dots,\epsilon_N^{\infty})$ and $\tilde{\mathbf e}^{\infty}=(\tilde\epsilon_1^{\infty},\dots,\tilde\epsilon_N^{\infty})$ respectively. Taking the limits in the both sides of
\begin{equation*}
\begin{split}
&\mathcal F(\tilde\Phi^{k_j})\varphi_i^{k_j}=\epsilon_i^{k_j}\varphi_i^{k_j}\\
&\mathcal F(\Phi^{k_j})\tilde\varphi_i^{k_j}=\tilde\epsilon_i^{k_j}\tilde\varphi_i^{k_j}
\end{split}
\qquad 1\leq i\leq N,
\end{equation*}
we obtain
\begin{equation*}
\begin{split}
&\mathcal F(\tilde\Phi^{\infty})\varphi_i^{\infty}=\epsilon_i^{\infty}\varphi_i^{\infty}\\
&\mathcal F(\Phi^{\infty})\tilde\varphi_i^{\infty}=\tilde\epsilon_i^{\infty}\tilde\varphi_i^{\infty}
\end{split}
\qquad 1\leq i\leq N.
\end{equation*}
Since ${\mathcal E}(\Phi^{\infty},\tilde\Phi^{\infty})=\mu$ and $\epsilon_i^{\infty},\tilde\epsilon_i^{\infty}\leq-\gamma,\ 1\leq i\leq N$ obviously hold, we can see that $(\Phi^{\infty},\tilde\Phi^{\infty})\in \Gamma_{\gamma,\mu}$, which completes the proof.
\end{proof}

\section{Fredholm property of Fr\'echet derivatives}\label{fifthsec}
In this section we prove that the Fr\'echet second derivatives of an auxiliary functional are decomposed into sums of an isomorphism and a compact operator.
Denote by
$$Y_1:=(\bigoplus_{i=1}^NH^2(\mathbb R^3))\bigoplus(\bigoplus_{i=1}^NH^2(\mathbb R^3))\bigoplus\mathbb R^N\bigoplus\mathbb R^N,$$
and
$$Y_2:=(\bigoplus_{i=1}^NL^2(\mathbb R^3))\bigoplus(\bigoplus_{i=1}^NL^2(\mathbb R^3))\bigoplus\mathbb R^N\bigoplus\mathbb R^N,$$
the direct sums of Banach spaces regarding $\bigoplus_{i=1}^NH^2(\mathbb R^3)$ and $\bigoplus_{i=1}^NL^2(\mathbb R^3)$ as real Banach spaces with respect to multiplication by real numbers.
Let us introduce an auxiliary functional.  We define a functional  $f:Y_1\to\mathbb R$ by
\begin{equation}\label{myeq5.0}
f(\Phi,\tilde\Phi,\mathbf e,\tilde{\mathbf e}):={\mathcal E}(\Phi,\tilde \Phi)-\sum_{i=1}^N\epsilon_i(\lVert \varphi_i\rVert^2-1)-\sum_{i=1}^N\tilde\epsilon_i(\lVert \tilde\varphi_i\rVert^2-1).
\end{equation}
We also define a bilinear form $\langle\langle \cdot,\cdot\rangle\rangle$ on $Y_1$ and $Y_2$ by
\begin{align*}
\langle\langle[\Phi^1,\tilde\Phi^1,\mathbf e^1,\tilde{\mathbf e}^1],[\Phi^2,\tilde\Phi^2,\mathbf e^2,\tilde{\mathbf e}^2]\rangle\rangle:=&2\sum_{i=1}^N\mathrm{Re}\, \langle\varphi_i^1,\varphi_i^2\rangle+2\sum_{i=1}^N\mathrm{Re}\, \langle\tilde\varphi_i^1,\tilde\varphi_i^2\rangle\\
&+\sum_{i=1}^N\epsilon_i^1\epsilon_i^2+\sum_{i=1}^N\tilde\epsilon_i^1\tilde\epsilon_i^2,
\end{align*}
for $[\Phi^j,\tilde\Phi^j,\mathbf e^j,\tilde{\mathbf e}^j]\in Y_j,\ j=1,2$.
Then the Fr\'echet derivative of $f$ is given by
$$df([\Phi^0,\tilde\Phi^0,\mathbf e^0,\tilde{\mathbf e}^0],[\Phi^1,\tilde\Phi^1,\mathbf e^1,\tilde{\mathbf e}^1])=\langle\langle[\Phi^1,\tilde\Phi^1,\mathbf e^1,\tilde{\mathbf e}^1],F(\Phi^0,\tilde\Phi^0,\mathbf e^0,\tilde{\mathbf e}^0)\rangle\rangle,$$
where $F:Y_1\to Y_2$ is defined by
\begin{align*}
&F(\Phi,\tilde\Phi,\mathbf e,\tilde{\mathbf e})\\
&\quad=\big[{}^t(F_1(\Phi,\tilde\Phi,\mathbf e),\dots,F_N(\Phi,\tilde\Phi,\mathbf e)),{}^t(F_1(\tilde\Phi,\Phi,\tilde{\mathbf e}),\dots,F_N(\tilde\Phi,\Phi,\tilde{\mathbf e})),\\
&\qquad(1-\lVert\varphi_1\rVert^2,\dots,1-\lVert\varphi_N\rVert^2),(1-\lVert\tilde\varphi_1\rVert^2,\dots,1-\lVert\tilde\varphi_N\rVert^2)\big].
\end{align*}
Here $F_i:(\bigoplus_{i=1}^NH^2(\mathbb R^3))\bigoplus(\bigoplus_{i=1}^NH^2(\mathbb R^3))\bigoplus\mathbb R^N\to L^2(\mathbb R^3)$ is given by
$$F_i(\Phi,\tilde\Phi,\mathbf e):=\mathcal F(\tilde \Phi)\varphi_i-\epsilon_i\varphi_i.$$

\begin{lem}\label{Fredholm}
For any $[\Phi',\tilde\Phi',\mathbf e',\tilde{\mathbf e}']\in Y_1$ satisfying $\epsilon_i', \tilde\epsilon_i'<0,\ 1\leq i\leq N$, the Fr\'echet derivative $F'(\Phi',\tilde\Phi',\mathbf e',\tilde{\mathbf e}')$ of $F(\Phi,\tilde\Phi,\mathbf e,\tilde{\mathbf e})$ at $[\Phi',\tilde\Phi',\mathbf e',\tilde{\mathbf e}']$ is written as
$$F'(\Phi',\tilde\Phi',\mathbf e',\tilde{\mathbf e}')=L+M,$$
where $\mathbf e'=(\epsilon_1'\dots,\epsilon_N')$, $\tilde{\mathbf e}'=(\tilde\epsilon_1'\dots,\tilde\epsilon_N')$, $L$ is an isomorphism of $Y_1$ onto $Y_2$ and $M$ is a compact operator.
\end{lem}

\begin{proof}
By the assumption clearly there exists a constant $\epsilon>0$ such that $\epsilon'_i,\tilde\epsilon'_i\leq-\epsilon$, $1\leq i\leq N$.
For a mapping $G(\Phi,\tilde\Phi):(\bigoplus_{i=1}^NH^2(\mathbb R^3))\bigoplus(\bigoplus_{i=1}^N H^2(\mathbb R^3))\to L^2(\mathbb R^3)$ we denote by $G_{\varphi_i}'=G_{\varphi_i}'(\Phi',\tilde\Phi'):H^2(\mathbb R^3)\to L^2(\mathbb R^3)$ the partial derivative
\begin{align*}
&G_{\varphi_i}'(\Phi',\tilde\Phi')h\\
&\quad:=\lim_{t\to0}[G(\varphi_1'\dots,\varphi_i'+th,\dots,\varphi_N',\tilde\Phi')-G(\varphi_1'\dots,\varphi_i',\dots,\varphi_N',\tilde\Phi')]/t,
\end{align*}
with respect to $\varphi_i$ at $(\Phi',\tilde\Phi')$, where $\Phi'={}^t(\varphi_1',\dots,\varphi'_N)$, $\tilde\Phi'={}^t(\tilde\varphi_1',\dots,\tilde\varphi'_N)$. The partial derivative $G_{\tilde\varphi_i}'=G_{\tilde\varphi_i}'(\Phi',\tilde\Phi'):H^2(\mathbb R^3)\to L^2(\mathbb R^3)$ with respect to $\tilde\varphi_i$ at $(\Phi',\tilde\Phi')$ is defined in the same way. For fixed $\mathbf e'$ and $\tilde{\mathbf e}'$ we shall consider the Fr\'echet derivative of the mapping $\check F(\Phi,\tilde \Phi):(\bigoplus_{i=1}^NH^2(\mathbb R^3))\bigoplus(\bigoplus_{i=1}^NH^2 (\mathbb R^3))\to (\bigoplus_{i=1}^NL^2(\mathbb R^3))\bigoplus(\bigoplus_{i=1}^N L^2(\mathbb R^3))$ given by
\begin{align*}
&\check F (\Phi,\tilde \Phi)\\
&\quad:=[{}^t(F_1(\Phi,\tilde\Phi,\mathbf e'),\dots,F_N(\Phi,\tilde\Phi,\mathbf e')),{}^t(F_1(\tilde\Phi,\Phi,\tilde{\mathbf e}'),\dots,F_N(\tilde\Phi,\Phi,\tilde{\mathbf e}'))].
\end{align*}
The Fr\'echet derivative
$$\check F'(\Phi',\tilde\Phi'):(\bigoplus_{i=1}^NH^2(\mathbb R^3))\bigoplus(\bigoplus_{i=1}^NH^2(\mathbb R^3))\to (\bigoplus_{i=1}^NL^2(\mathbb R^3))\bigoplus(\bigoplus_{i=1}^NL^2(\mathbb R^3)),$$
of $\check F$ at $(\Phi',\tilde\Phi')$ can be expressed as a $2N\times 2N$ matrix of operators from $H^2(\mathbb R^3)$ to $L^2(\mathbb R^3)$ as
\begin{equation}\label{myeq5.1}
\check F'(\Phi',\tilde\Phi')=
\begin{pmatrix}
K(\tilde\Phi',\mathbf e') & T^{\Phi',\tilde\Phi'}\\
T^{\tilde\Phi',\Phi'} & K(\Phi',\tilde{\mathbf e}')
\end{pmatrix},
\end{equation}
where $K(\tilde\Phi',\mathbf e')$ is a diagonal matrix defined by
$$K(\tilde\Phi',\mathbf e'):=\mathrm{diag}\, [\mathcal F(\tilde\Phi')-\epsilon_1',\dots,\mathcal F(\tilde \Phi')-\epsilon_N'],$$
and the $N\times N$ matrix $T^{\Phi',\tilde\Phi'}$ of operators is given by
\begin{align*}
T_{ij}^{\Phi',\tilde\Phi'}=[F_i(\Phi,\tilde \Phi,\mathbf e')]'_{\tilde\varphi_j}=\hat S_{ij}^{\Phi',\tilde\Phi'}+\check S_{ij}^{\Phi',\tilde\Phi'}- Q_{ij}^{\Phi',\tilde\Phi'}-\check S_{ji}^{\tilde\Phi',\Phi'}.
\end{align*}
Here
\begin{align*}
(\hat S_{ij}^{\Phi',\tilde\Phi'}w)(x)&:=\left(\int|x-y|^{-1}\overline{\tilde\varphi_j'}(y)w(y)dy\right)\varphi_i'(x),\\
(\check S_{ij}^{\Phi',\tilde\Phi'}w)(x)&:=\left(\int|x-y|^{-1}\overline{w}(y)\tilde\varphi_j'(y)dy\right)\varphi_i'(x),\\
(Q_{ij}^{\Phi',\tilde\Phi'}w)(x)&:=\left(\int|x-y|^{-1}\overline{\tilde\varphi_j'}(y)\varphi_i'(y)dy\right)w(x).\\
\end{align*}
Let us define the matrices $\hat S^{\Phi',\tilde\Phi'}$, $\check S^{\Phi',\tilde\Phi'}$ and $Q^{\Phi',\tilde\Phi'}$ by the matrix elements $\hat S_{ij}^{\Phi',\tilde\Phi'}$, $\check S_{ij}^{\Phi',\tilde\Phi'}$ and $Q_{ij}^{\Phi',\tilde\Phi'}$ respectively.
We can rewrite \eqref{myeq5.1} as
$$\check F'(\Phi',\tilde\Phi')=\mathcal K+\mathcal T,$$
with
$$\mathcal K:=
\begin{pmatrix}
K(\tilde\Phi',\mathbf e') & 0\\
0 & K(\Phi',\tilde{\mathbf e}')
\end{pmatrix},\ \mathcal T:=
\begin{pmatrix}
0 & T^{\Phi',\tilde\Phi'}\\
T^{\tilde\Phi',\Phi'} & 0
\end{pmatrix}.$$
The matrices $\hat{\mathcal S}$, $\check{\mathcal S}$ and $\mathcal Q$ are defined replacing $T^{\Phi',\tilde\Phi'}$ in $\mathcal T$ by $\hat S^{\Phi',\tilde\Phi'}$, $\check S^{\Phi',\tilde\Phi'}$ and $Q^{\Phi',\tilde\Phi'}$ respectively.
Then we have
$$\mathcal T=\hat{\mathcal S}+\check{\mathcal S}-\mathcal Q-{}^t\check{\mathcal S}.$$
On the other hand $\mathcal K$ is decomposed as
$$\mathcal K=\mathcal H+\mathcal R-\mathcal S,$$
where
\begin{align*}
\mathcal H&:=\mathrm{diag}\, [h-\epsilon_1',\dots,h-\epsilon_N',h-\tilde\epsilon_1',\dots,h-\tilde\epsilon_N'],\\
\mathcal R&:=\mathrm{diag}\, [R^{\tilde\Phi'},\dots,R^{\tilde\Phi'},R^{\Phi'},\dots,R^{\Phi'}],\\
\mathcal S&:=\mathrm{diag}\, [S^{\tilde\Phi'},\dots,S^{\tilde\Phi'},S^{\Phi'},\dots,S^{\Phi'}].
\end{align*}
Since $S^{\tilde\Phi'}_{ii}$, $S^{\Phi'}_{ii}$, $\hat S^{\Phi',\tilde\Phi'}_{ij}$ and $\check S^{\Phi',\tilde\Phi'}_{ij}$ are integral operators of the Hilbert-Schmidt type, they are compact. Thus $\mathcal S$, $\hat{\mathcal S}$, $\check {\mathcal S}$ and ${}^t\check {\mathcal S}$ are compact operators.

We shall show that $\mathcal R-\mathcal Q$ is a positive definite operator as an operator on the Hilbert space $(\bigoplus_{i=1}^NL^2(\mathbb R^3))\bigoplus(\bigoplus_{i=1}^NL^2(\mathbb R^3))$. Set $W:={}^t(w_1,\dots,w_N,\tilde w_1,\dots,\tilde w_N)\newline\in(\bigoplus_{i=1}^NL^2(\mathbb R^3))\bigoplus(\bigoplus_{i=1}^NL^2(\mathbb R^3))$. Then we have
\begin{equation}\label{mye5.2}
\begin{split}
\langle W,(\mathcal R-\mathcal Q)W\rangle&=\sum_{i=1}^N\langle w_i,R^{\tilde\Phi'}w_i\rangle+\sum_{i=1}^N\langle \tilde w_i,R^{\Phi'}\tilde w_i\rangle\\
&\quad-\sum_{i,j=1}^N\langle w_i,Q_{ij}^{\Phi',\tilde\Phi'}\tilde w_j\rangle-\sum_{i,j=1}^N\langle \tilde w_i,Q_{ij}^{\tilde\Phi',\Phi'}w_j\rangle\\
&=\sum_{i,j=1}^N\{\langle w_i,Q_{jj}^{\tilde \Phi'}w_i\rangle+\langle \tilde w_j,Q_{ii}^{\Phi'}\tilde w_j\rangle\\
&\quad-\langle w_i,Q_{ij}^{\Phi',\tilde\Phi'}\tilde w_j\rangle-\langle \tilde w_j,Q_{ji}^{\tilde\Phi',\Phi'}w_i\rangle\}.
\end{split}
\end{equation}
On the other hand we have
\begin{equation}\label{mye5.3}
\begin{split}
&\int|x-y|^{-1}|w_i(x)\tilde\varphi_j'(y)-\tilde w_j(x)\varphi_i'(y)|^2dxdy\\
&\quad=\langle w_i,Q_{jj}^{\tilde \Phi'}w_i\rangle+\langle \tilde w_j,Q_{ii}^{\Phi'}\tilde w_j\rangle-\langle w_i,Q_{ij}^{\Phi',\tilde\Phi'}\tilde w_j\rangle-\langle \tilde w_j,Q_{ji}^{\tilde\Phi',\Phi'}w_i\rangle.
\end{split}
\end{equation}
Since the left-hand side is positive, the right-hand side is also positive. Therefore, comparing \eqref{mye5.2} with \eqref{mye5.3} we can see that $\mathcal R-\mathcal Q$ is a positive definite operator.

Next we shall consider $\mathcal H$. We denote the resolution of identity of $h$ by $E(\lambda)$. Then we can  decompose $h$ as
$$h=hE(-\epsilon/2)+h(1-E(-\epsilon/2)).$$
Thus $\mathcal H$ is decomposed as $\mathcal H=\mathcal H_1+\mathcal H_2$, where
\begin{align*}
\mathcal H_1:=&\mathrm{diag}\, [h(1-E(-\epsilon/2))-\epsilon_1',\dots,h(1-E(-\epsilon/2))-\epsilon_N',\\
&h(1-E(-\epsilon/2))-\tilde\epsilon_1',\dots,h(1-E(-\epsilon/2))-\tilde\epsilon_N'],
\end{align*}
and
$$\mathcal H_2:=\mathrm{diag}\, [hE(-\epsilon/2),\dots,hE(-\epsilon/2),hE(-\epsilon/2),\dots,hE(-\epsilon/2)].$$
Since $\epsilon_i', \tilde\epsilon_i'\leq-\epsilon,\ 1\leq i\leq N$ we have $h(1-E(-\epsilon/2))-\epsilon_i'\geq\epsilon/2$, and $h(1-E(-\epsilon/2))-\tilde\epsilon_i'\geq\epsilon/2$, so that $\mathcal H_1\geq \epsilon/2$. As for $\mathcal H_2$, $\inf\sigma_{ess}(h)=0$ implies that $hE(-\epsilon/2)$ is a compact operator. Thus $\mathcal H_2$ is a compact operator.

The Fr\'echet derivative $\check F'(\Phi',\tilde\Phi')$ is written as
$$\check F'(\Phi',\tilde\Phi')=\mathcal H_1+\mathcal H_2+\mathcal R-\mathcal S + \hat{\mathcal S}+\check{\mathcal S}-\mathcal Q-{}^t\check{\mathcal S}=\mathcal L+\mathcal M,$$
where $\mathcal L:=\mathcal H_1+\mathcal R-\mathcal Q$ and $\mathcal M:=\mathcal H_2-\mathcal S + \hat{\mathcal S}+\check{\mathcal S}-{}^t\check{\mathcal S}$. Since $\mathcal H_1\geq\epsilon/2$ and $\mathcal R-\mathcal Q\geq0$, we have $\mathcal L\geq \epsilon/2$, and thus $\mathcal L$ is invertible. Since $\mathcal L$ can be regarded as a self-adjoint operator in $(\bigoplus_{i=1}^NL^2(\mathbb R^3))\bigoplus(\bigoplus_{i=1}^NL^2(\mathbb R^3))$, we can see that $\mathrm{Ran}\, \mathcal L=(\bigoplus_{i=1}^NL^2(\mathbb R^3))\bigoplus(\bigoplus_{i=1}^NL^2(\mathbb R^3))$ and it is an isomorphism of $(\bigoplus_{i=1}^NH^2(\mathbb R^3))\bigoplus(\bigoplus_{i=1}^NH^2(\mathbb R^3))$ onto $(\bigoplus_{i=1}^NL^2(\mathbb R^3))\bigoplus(\bigoplus_{i=1}^NL^2(\mathbb R^3))$. Moreover, since each term in $\mathcal M$ is a compact operator, $\mathcal M$ is also a compact operator.

For fixed $\Phi',\tilde\Phi'$ we set
$$\hat F(\mathbf e,\tilde{\mathbf e}):={}^t(F_1(\Phi',\tilde\Phi',\mathbf e),\dots,F_N(\Phi',\tilde\Phi',\mathbf e),F_1(\tilde\Phi',\Phi',\tilde{\mathbf e}),\dots,F_N(\tilde\Phi',\Phi',\tilde{\mathbf e})).$$
Then we obtain
\begin{align*}
&F'(\Phi',\tilde\Phi',\mathbf e',\tilde{\mathbf e}')[\Phi,\tilde\Phi,\mathbf e,\tilde{\mathbf e}]\\
&\quad=[\check F'(\Phi',\tilde\Phi')[\Phi,\tilde\Phi]+\hat F'(\mathbf e',\tilde{\mathbf e}')[\mathbf e,\tilde{\mathbf e}],\\
&\qquad-2\mathrm{Re}\, \langle \varphi_1,\varphi_1'\rangle,\dots,-2\mathrm{Re}\, \langle \varphi_N,\varphi_N'\rangle,-2\mathrm{Re}\, \langle \tilde\varphi_1,\tilde\varphi_1'\rangle,\dots,-2\mathrm{Re}\, \langle \tilde\varphi_N,\tilde\varphi_N'\rangle]\\
&\quad=L[\Phi,\tilde\Phi,\mathbf e,\tilde{\mathbf e}]+M[\Phi,\tilde\Phi,\mathbf e,\tilde{\mathbf e}],
\end{align*}
where
\begin{align*}
&L[\Phi,\tilde\Phi,\mathbf e,\tilde{\mathbf e}]:=[\mathcal L[\Phi,\tilde\Phi],\mathbf e,\tilde{\mathbf e}],\\
&M[\Phi,\tilde\Phi,\mathbf e,\tilde{\mathbf e}]\\
&\quad:=[\mathcal M[\Phi,\tilde\Phi]-[\mathbf e\Phi',\tilde{\mathbf e}\tilde\Phi'],-2\mathrm{Re}\, \langle \varphi_1,\varphi_1'\rangle-\epsilon_1,\dots,-2\mathrm{Re}\, \langle \varphi_N,\varphi_N'\rangle-\epsilon_N,\\
&\qquad-2\mathrm{Re}\, \langle \tilde\varphi_1,\tilde\varphi_1'\rangle-\tilde\epsilon_1,\dots,-2\mathrm{Re}\, \langle \tilde\varphi_N,\tilde\varphi_N'\rangle-\tilde\epsilon_N].\\
\end{align*}
Here $\mathbf e\Phi':={}^t(\epsilon_1\varphi_1',\dots,\epsilon_N\varphi_N')$. We can easily see that  $L$ is an isomorphism and $M$ is a compact operator, which completes the proof.
\end{proof}

\section{\L ojasiewicz inequality}\label{sixthsec}
The \L ojasiewcz inequality for functionals that satisfy a certain condition is crucial for the proof of the convergence of SCF sequences. Let us denote by $\lVert \cdot\rVert_X$ the norm in a Banach space $X$.

\begin{dfn}
Let $X$ and $Y$ be real Banach spaces and $O$ an open subset of $X$. The mapping $F:O\to Y$ is said to be real-analytic on $O$ if the following conditions are fulfilled:
\begin{itemize}
\item[(i)] For each $x\in O$ there exist Fr\'echet derivatives of arbitrary orders $d^mF(x,\dots)$.
\item[(ii)] For each $x\in O$ there exists $\delta>0$ such that for any $h\in X$ satisfying $\lVert h\rVert_X<\delta$ one has
\begin{equation*}
F(x+h)=\sum_{m=0}^{\infty}\frac{1}{m!}d^mF(x,h^m),
\end{equation*}
(the convergence being locally uniform and absolute), where $h^m := [h, \dots ,h]$ ($m$-times).
\end{itemize}
\end{dfn}

\begin{lem}\label{Lojasiewicz}
Let $Z$ be a real Hilbert space equipped with an inner product $\langle\langle\cdot,\cdot\rangle\rangle$ and the norm $\lVert \cdot\rVert_Z:=\langle\langle\cdot,\cdot\rangle\rangle^{1/2}$. Let $X$ be a dense subspace of $Z$ and assume that $X$ is a real Banach space with respect to another norm $\lVert \cdot\rVert_X$ such that $\lVert x\rVert_Z\leq\lVert x\rVert_X$ for any $x\in X$. Moreover, let $f(x)$ be a real-analytic functional in $X$ and $x^c$ a critical point of $f(x)$. Suppose that there exists a real-analytic mapping $F(x):X\to Z$ such that
\begin{itemize}
\item[(f1)] $df(x,y)=\langle\langle y,F(x)\rangle\rangle$ for any $x, y\in X$.
\item[(f2)] $F'(x^c)=L+M$, where $L$ is an isomorphism of $X$ onto $Z$ and $M$ is a compact operator.
\item[(f3)] $F'(x^c)$ is a selfadjoint operator with the domain $X$, when it is regarded as an operator in $Z$.
\end{itemize}
Then there exists a constant $\kappa>0$, $\theta\in(0,1/2]$ and a neighborhood $U(x^c)$ of $x^c$ such that
\begin{equation}\label{myeq6.0}
|f(x)-f(x^c)|^{1-\theta}\leq\kappa\lVert F(x)\rVert_{Z},
\end{equation}
for any $x\in U(x^c)$.
\end{lem}

For the proof of Lemma \ref{Lojasiewicz} we need the following real-analytic version of the implicit function theorem.

\begin{lem}[{\cite[Proposition 2.1]{FNSS} see also \cite[Lemma 3R]{FNSS2}}]\label{implicit}
Let $\mathcal X,\mathcal Y,\mathcal Z$ be real Banach spaces, $O\subset \mathcal X\times \mathcal Y$ an open set and $[x_0,y_0]\in O$. Let $F: O\to \mathcal Z$ be a real-analytic mapping such that $[F_y'(x_0,y_0)]^{-1}:\mathcal Z\to\mathcal Y$ exists and $F(x_0,y_0)=0$. Then there exist a neighborhood $U(x_0)$ in $\mathcal X$ of the point $x_0$ and a neighborhood $U(y_0)$ in $\mathcal Y$ of the point $y_0$  such that $U(x_0)\times U(y_0)\subset O$ and there exists one and only one mapping $y: U(x_0)\to U(y_0)$ for which $F(x,y(x))=0$ on $U(x_0)$. Moreover, $y$ is a real-analytic mapping on $U(x_0)$.
\end{lem}

\begin{proof}[Proof of Lemma \ref{Lojasiewicz}]
Due to the decomposition $F'(x^c)=L+M$, $F'(x^c)$ is a Fredholm operator (see e.g. \cite[Proof of Theorem 2.1]{As2}). Thus $X_1:=\mathrm{Ker}\, (F'(x^c))$ is finite-dimensional. Set $X_2:=X_1^{\perp}\cap X$, where $X_1^{\perp}$ is the orthogonal subspace of $X_1$ in $Z$. Then $X_1$ and $X_2$ are closed subspaces of $X$ and we have $X=X_1\bigoplus X_2$. In addition, $F'(x^c)$ is an isomorphism of $X_2$ onto a closed subspace $\tilde Z:=F'(x^c)(X_2)$ of $Z$. We write $x=[x_1,x_2],\ x_i \in X_i\ (i=1, 2)$ correspondingly to the decomposition $X=X_1\bigoplus X_2$. Let us denote the norm $\lVert\cdot\rVert_X$ restricted to $X_2$ by $\lVert \cdot\rVert_{X_2}$ and the norm $\lVert\cdot\rVert_Z$ restricted to $\tilde Z$ by $\lVert \cdot\rVert_{\tilde Z}$ with which $X_2$ and $\tilde Z$ are regarded as Banach spaces.

If $X_1=\{0\}$, then by the open mapping theorem $F'(x^c)^{-1}:\tilde Z\to X$ is continuous, and therefore, there exists a constant $\check C_1>0$ such that
\begin{equation}\label{myeq6.1.1}
\lVert\tilde x\rVert_X=\lVert F'(x^c)^{-1}F'(x^c)\tilde x\rVert_{X}\leq \check C_1\lVert F'(x^c)\tilde x\rVert_{\tilde Z}=\check C_1\lVert F'(x^c)\tilde x\rVert_{Z},
\end{equation}
for any $\tilde x\in X$.
Since by the definition of the Fr\'echet derivative and $F(x^c)=0$ we have $F(x)=F(x^c)+F'(x^c)(x-x^c)+o(\lVert x-x^c\rVert_X)=F'(x^c)(x-x^c)+o(\lVert x-x^c\rVert_X)$, using \eqref{myeq6.1.1} we can see that there exists a neighborhood $\hat U(x^c)$ of $x^c$ and constants $0<\tau<\check C_1^{-1}$, $\check C_2>0$ such that
\begin{equation}\label{myeq6.2}
\lVert F(x)\rVert_Z\geq \check C_1^{-1}\lVert x-x^c\rVert_X-\tau\lVert x-x^c\rVert_X\geq \check C_2\lVert x-x^c\rVert_X,
\end{equation}
for any $x\in \hat U(x^c)$. On the other hand since $x^c$ is a critical point of $f(x)$, by the Taylor formula (see e.g. \cite[Theorem 4.A]{Ze}) we have
$$f(x)=f(x^c)+\int_0^1(1-t)d^2f(x^c+t(x-x^c),(x-x^c)^2)dt.$$
Hence we can see that there exists a constant $\check C_3>0$ such that 
\begin{equation}\label{myeq6.1}
|f(x)-f(x^c)|\leq \check C_3\lVert x-x^c\rVert^2_X,
\end{equation}
for any $x\in \hat U(x^c)$. From \eqref{myeq6.1} and \eqref{myeq6.2} it is seen that \eqref{myeq6.0} holds with $U(x^c)=\hat U(x^c)$, $\kappa=\check C_2^{-1}\check C_3^{1/2}$ and $\theta=1/2$ if $X_1=\{0\}$.

If $X_1\neq \{0\}$, applying Lemma \ref{implicit} to $P_{\tilde Z}\circ F$ near $x^c$ with $\mathcal X=X_1$, $\mathcal Y=X_2$ and $\mathcal Z=\tilde Z$ it follows that there exists neighborhoods $\hat U(x^c_1)$ of $x^c_1$, $\hat U(x^c_2)$ of $x^c_2$ and a real-analytic mapping $\omega:\hat U(x^c_1)\to \hat U(x^c_2)$ such that $x=[x_1,x_2]\in \hat U(x^c_1)\times \hat U(x^c_2)$ satisfies $(P_{\tilde Z}\circ F)(x)=P_{\tilde Z}F(x)=0$ if and only if $x_2=\omega(x_1)$, where $P_{\tilde Z}$ is the orthogonal projection from $Z$ onto $\tilde Z$. Moreover, since $F'(x^c)$ is selfadjoint, we have $X_1=\mathrm{Ker}\, (F'(x^c))=\mathrm{Ker}\, (F'(x^c)^*)=\tilde Z^{\perp}$, where $\tilde Z^{\perp}$ is the orthogonal subspace of $\tilde Z$ in $Z$. Let $\nu:=\mathrm{dim}\, X_1$ and $\{v_1,\dots,v_{\nu}\}$ be a basis of $X_1$. Set $\mathbf v:=(v_1,\dots,v_{\nu})$. We write $\mathbf t\cdot \mathbf v:=\sum_{j=1}^{\nu}t_jv_j$ for $\mathbf t=(t_1,\dots,t_{\nu})\in\mathbb R^{\nu}$. Then any  $x_1\in\hat U(x_1^c)$ is expressed as $x_1=x^c_1+\mathbf t \cdot \mathbf v$, and $f(x_1^c+\mathbf t\cdot\mathbf v,\omega(x_1^c+\mathbf t\cdot\mathbf v))$ is a real-analytic function of $\mathbf t\in\mathbb R^{\nu}$ since $f$ and $\omega$ are real-analytic. Thus applying the \L ojasiewicz inequality in finite-dimensional space we can see that there exist constants $\kappa_1,\kappa_2>0$, $\theta\in (0,1/2]$ and a neighborhood $\check U(x_1^c)$ of $x_1^c$ such that for any $x_1\in \check U(x_1^c)$
\begin{equation}\label{myeq6.3}
\begin{split}
&|f(x_1,\omega(x_1))-f(x_1^c,x_2^c)|^{1-\theta}\\
&\quad=|f(x_1,\omega(x_1))-f(x_1^c,\omega(x_1^c))|^{1-\theta}\\
&\quad=|f(x_1^c+\mathbf t\cdot \mathbf v,\omega(x_1^c+\mathbf t\cdot \mathbf v))-f(x_1^c,\omega(x_1^c))|^{1-\theta}\\
&\quad=\kappa_1\sum_{j=1}^{\nu}\big|\langle\langle v_j,F(x_1^c+\mathbf t\cdot \mathbf v,\omega(x_1^c+\mathbf t\cdot \mathbf v))\rangle\rangle\\
&\qquad+\langle\langle \omega'(x_1^c+\mathbf t\cdot\mathbf v)[v_j],F(x_1^c+\mathbf t\cdot \mathbf v,\omega(x_1^c+\mathbf t\cdot \mathbf v))\rangle\rangle\big|\\
&\quad\leq\kappa_2\lVert F(x_1^c+\mathbf t\cdot \mathbf v,\omega(x_1^c+\mathbf t\cdot \mathbf v))\rVert_{Z}\\
&\quad=\kappa_2\lVert F(x_1,\omega(x_1))\rVert_Z,
\end{split}
\end{equation}
where we used that $\lVert\omega'(x_1^c+\mathbf t\cdot\mathbf v)[v_j]\rVert_{Z}\leq\lVert\omega'(x_1^c+\mathbf t\cdot\mathbf v)[v_j]\rVert_{X}\leq C$ for a constant $C>0$ independent of $\mathbf t$ such that $x_1^c+\mathbf t\cdot \mathbf v\in \check U(x_1^c)$.

By the open mapping theorem $[P_{\tilde Z}F'_{x_2}(x_1^c,x_2^c)]^{-1}:\tilde Z\to X_2$ is continuous (note that $(P_{\tilde Z}\circ F)'(x)=P_{\tilde Z}F'(x)$). Thus there exists a constant $\check C_4>0$ such that
\begin{equation}\label{myeq6.3.1}
\begin{split}
\lVert \tilde x_2\rVert_{X_2}&=\lVert [P_{\tilde Z}F'_{x_2}(x_1^c,x_2^c)]^{-1}P_{\tilde Z}F'_{x_2}(x_1^c,x_2^c)\tilde x_2\rVert_{X_2}\\
&\leq \check C_4\lVert P_{\tilde Z}F'_{x_2}(x_1^c,x_2^c)\tilde x_2\rVert_{\tilde Z}=\check C_4\lVert P_{\tilde Z}F'_{x_2}(x_1^c,x_2^c)\tilde x_2\rVert_{Z},
\end{split}
\end{equation}
for any $\tilde x_2\in X_2$.
Since $F(x)$ is a real-analytic mapping, choosing $\hat U(x^c_1)$ and $\hat U(x^c_2)$ small enough and using \eqref{myeq6.3.1} we can see that there exists $\check C_5>0$ such that
\begin{equation}\label{myeq6.4}
\begin{split}
&\lVert P_{\tilde Z}F'_{x_2}(x_1,\omega(x_1))\tilde x_2\rVert_Z\\
&\quad\geq\lVert P_{\tilde Z}F'_{x_2}(x_1^c,x_2^c)\tilde x_2\rVert_Z-\lVert (P_{\tilde Z}F'_{x_2}(x_1,\omega(x_1))-P_{\tilde Z}F'_{x_2}(x_1^c,x_2^c))\tilde x_2\rVert_Z\\
&\quad\geq \check C_5\lVert\tilde x_2\rVert_{X_2},
\end{split}
\end{equation}
for any $x_1\in \hat U(x_1^c)$ and $\tilde x_2\in X_2$. Moreover, by the definition of the Fr\'echet derivative and $P_{\tilde Z}F(x_1,\omega(x_1))=0$ we have
\begin{equation}\label{myeq6.5}
\begin{split}
&P_{\tilde Z}F(x_1,x_2)\\
&\quad=P_{\tilde Z}F(x_1,\omega(x_1))+P_{\tilde Z}F'_{x_2}(x_1,\omega(x_1))(x_2-\omega(x_1))+o(\lVert x_2-\omega(x_1)\rVert_{X_2})\\
&\quad=P_{\tilde Z}F'_{x_2}(x_1,\omega(x_1))(x_2-\omega(x_1))+o(\lVert x_2-\omega(x_1)\rVert_{X_2}),
\end{split}
\end{equation}
for any $x_1\in \hat U(x_1^c)$ and $x_2\in \hat U(x_2^c)$. It follows from \eqref{myeq6.4} and \eqref{myeq6.5} that choosing smaller $\hat U(x_1^c)$ and $\hat U(x_2^c)$ further there exists $\check C_6>0$ such that
\begin{equation}\label{myeq6.6}
\lVert P_{\tilde Z}F(x_1,x_2)\rVert_Z\geq \check C_6\lVert x_2-\omega(x_1)\rVert_{X_2},
\end{equation}
for any  $x_1\in \hat U(x_1^c)$ and $x_2\in \hat U(x_2^c)$. Moreover, using \eqref{myeq6.6} we can see that there exists $\check C_7>0$ such that
\begin{equation}\label{myeq6.6.1}
\begin{split}
\lVert F(x_1,\omega(x_1))\rVert_Z&\leq \lVert F(x_1,x_2)\rVert_Z+ \lVert F(x_1,x_2)-F(x_1,\omega(x_1))\rVert_Z\\
&\leq\lVert F(x_1,x_2)\rVert_Z+\check C_7\lVert x_2-\omega(x_1)\rVert_{X_2}\\
&\leq (1+\check C_7\check C_6^{-1})\lVert F(x_1,x_2)\rVert_Z.
\end{split}
\end{equation}

On the other hand, since $P_{\tilde Z}F(x_1,\omega(x_1))=0$, we have
\begin{equation*}
\begin{split}
&f(x_1,x_2)-f(x_1,\omega(x_1))\\
&\quad=\langle\langle x_2-\omega(x_1),P_{\tilde Z^{\perp}}F(x_1,\omega(x_1))\rangle\rangle+O(\lVert x_2-\omega(x_1)\rVert_{X_2}^2),
\end{split}
\end{equation*}
where $P_{\tilde Z^{\perp}}$ is the orthogonal projection onto $\tilde Z^{\perp}$. Since $\tilde Z^{\perp}=X_1$ and $X_1\perp X_2$ with respect to the inner product $\langle\langle\cdot,\cdot\rangle\rangle$ in $Z$, the first term in the right-hand side vanishes.
Thus there exists a constant $\check C_8>0$ such that
\begin{equation}\label{myeq6.7}
|f(x_1,x_2)-f(x_1,\omega(x_1))|\leq \check C_8\lVert x_2-\omega(x_1)\rVert_{X_2}^2,
\end{equation}
for any $x_1\in \hat U(x_1^c)$ and $x_2\in \hat U(x_2^c)$. Combining \eqref{myeq6.6} and \eqref{myeq6.7} we obtain
\begin{equation*}
|f(x_1,x_2)-f(x_1,\omega(x_1))|^{1/2}\leq \check C_6^{-1}\check C_8^{1/2}\lVert F(x_1,x_2)\rVert_Z.
\end{equation*}

It follows from \eqref{myeq6.3}, \eqref{myeq6.6.1} and \eqref{myeq6.7} that for $x\in U(x^c):=(\check U(x_1^c)\cap\hat U(x_1^c))\times \hat U(x_2^c)$ we have
\begin{align*}
&|f(x_1,x_2)-f(x^c)|^{1-\theta}\\
&\quad=|f(x_1,x_2)-f(x_1,\omega(x_1))+f(x_1,\omega(x_1))-f(x^c)|^{1-\theta}\\
&\quad\leq 2^{1-\theta}(|f(x_1,x_2)-f(x_1,\omega(x_1))|^{1-\theta}+|f(x_1,\omega(x_1))-f(x^c)|^{1-\theta})\\
&\quad\leq 2^{1-\theta}(|f(x_1,x_2)-f(x_1,\omega(x_1))|^{1/2}+|f(x_1,\omega(x_1))-f(x^c)|^{1-\theta})\\
&\quad\leq 2^{1-\theta}(\check C_6^{-1}\check C_8^{1/2}+\kappa_2(1+\check C_7\check C_6^{-1}))\lVert F(x_1,x_2)\rVert_Z,
\end{align*}
where in the third step we assume $|f(x_1,x_2)-f(x_1,\omega(x_1))|<1$, which holds if we choose sufficiently small $U(x^c)$.
Thus \eqref{myeq6.0} holds with $\kappa=2^{1-\theta}(\check C_6^{-1}\check C_8^{1/2}+\kappa_2(1+\check C_7\check C_6^{-1}))$, which completes the proof.
\end{proof}

\section{Proof of the main theorems}\label{seventhsec}
For the proof of Theorem \ref{main} the following lemma about convergence of positive term series is needed.
\begin{lem}\label{series}
Let $(\alpha_1,\alpha_2,\dots)$ be a sequence of real numbers such that $\alpha_k>0$ for any $k\geq 1$ and $\sum_{k=1}^{\infty}\frac{\alpha_{k+1}^2}{\alpha_k}$ converges. Then $\sum_{k=1}^{\infty}\alpha_k$ converges.
\end{lem}

\begin{proof}
Let $k_0\in\mathbb N$ be a fixed number. Then by the Cauchy-Schwarz inequality we have
$$\sum_{k=1}^{k_0}\alpha_{k+1}=\sum_{k=1}^{k_0}\frac{\alpha_{k+1}}{\alpha_k^{1/2}}\alpha_k^{1/2}\leq\left(\sum_{k=1}^{k_0}\frac{\alpha_{k+1}^2}{\alpha_k}\right)^{1/2}\left(\sum_{k=1}^{k_0}\alpha_k\right)^{1/2}.$$
Hence we have
$$\sum_{k=1}^{k_0}\alpha_{k}\leq \alpha_1+\sum_{k=1}^{k_0}\alpha_{k+1}\leq \alpha_1+\left(\sum_{k=1}^{k_0}\frac{\alpha_{k+1}^2}{\alpha_k}\right)^{1/2}\left(\sum_{k=1}^{k_0}\alpha_k\right)^{1/2}.$$
Dividing both sides by $\left(\sum_{k=1}^{k_0}\alpha_k\right)^{1/2}$ we obtain
$$\left(\sum_{k=1}^{k_0}\alpha_k\right)^{1/2}\leq \alpha_1\left(\sum_{k=1}^{k_0}\alpha_k\right)^{-1/2}+\left(\sum_{k=1}^{k_0}\frac{\alpha_{k+1}^2}{\alpha_k}\right)^{1/2}\leq \alpha_1^{1/2}+\left(\sum_{k=1}^{k_0}\frac{\alpha_{k+1}^2}{\alpha_k}\right)^{1/2}.$$
Since $\sum_{k=1}^{\infty}\frac{\alpha_{k+1}^2}{\alpha_k}$ converges, the right-hand side is bounded by a constant $C>0$ independent of $k_0$, and therefore, we have
$$\sum_{k=1}^{k_0}\alpha_k\leq C^2.$$
Since $k_0$ was arbitrary, this implies that $\sum_{k=1}^{\infty}\alpha_k$ is convergent and
$$\sum_{k=1}^{\infty}\alpha_k\leq C^2,$$
which completes the proof.
\end{proof}

We also need a bound of the $H^2$ norm of differences of solutions to a sequence of equations by the $L^2$ norm.
\begin{lem}\label{reduction}
Let $\zeta>0$ be a constant and $\Xi^k={}^t(\xi_1^k,\dots,\xi_N^k)\in\mathcal W,\ k=0,1,\dots$ be a sequence satisfying
$$\mathcal F(\Xi^{k-1})\xi_i^{k}=\sum_{j=1}^N\epsilon_{ij}^k\xi_j^k,$$
with some constants $\epsilon^k_{ij},\ 1\leq i,j\leq N$ such that $|\epsilon^k_{ij}|\leq \zeta$, $1\leq i,j\leq N$ for any $k\geq 1$. Then there exists a constant $\beta_{\zeta}>0$ independent of $k$ such that
\begin{align*}
&\lVert \Xi^{k+1}-\Xi^{k-1}\rVert_{\bigoplus_{i=1}^NH^2(\mathbb R^3)}\\
&\quad\leq \beta_{\zeta}(\lVert \Xi^{k}-\Xi^{k-2}\rVert_{\bigoplus_{i=1}^NL^2(\mathbb R^3)}+\lVert \Xi^{k+1}-\Xi^{k-1}\rVert_{\bigoplus_{i=1}^NL^2(\mathbb R^3)}),
\end{align*}
for any $k\geq 2$.
\end{lem}

\begin{proof}
First note that by the same proof as that of Lemma \ref{Hbound} we can see that there exists a constant $\tilde C_{\zeta}'>0$ such that 
\begin{equation}\label{myeq7.0.0}
\lVert\nabla\xi_i^k\rVert\leq\tilde C_{\zeta}',\ 1\leq i\leq N,
\end{equation}
for any $k\geq 0$. It follows from the equations
\begin{equation}\label{myeq7.0}
\begin{split}
\mathcal F(\Xi^{k-2})\xi_i^{k-1}&=\sum_{j=1}^N\epsilon_{ij}^{k-1}\xi_j^{k-1},\\
\mathcal F(\Xi^{k})\xi_i^{k+1}&=\sum_{j=1}^N\epsilon_{ij}^{k+1}\xi_j^{k+1},
\end{split}
\end{equation}
the Hardy inequality and \eqref{myeq7.0.0} that there exists a constant $\tilde\beta_{\zeta}>0$ independent of $k$ such that
\begin{equation}\label{myeq7.0.1}
\begin{split}
&\lVert h(\xi_i^{k+1}-\xi_i^{k-1})\rVert\\
&\quad=\lVert\mathcal G(\Xi^{k})\xi_i^{k+1}-\sum_{j=1}^N\epsilon_{ij}^{k+1}\xi_j^{k+1}-\mathcal G(\Xi^{k-2})\xi_i^{k-1}-\sum_{j=1}^N\epsilon_{ij}^{k-1}\xi_j^{k-1}\rVert\\
&\quad\leq\tilde\beta_{\zeta}(\lVert \Xi^{k}-\Xi^{k-2}\rVert_{\bigoplus_{i=1}^NL^2(\mathbb R^3)}+\lVert \Xi^{k+1}-\Xi^{k-1}\rVert_{\bigoplus_{i=1}^NL^2(\mathbb R^3)}\\
&\qquad+\sum_{j=1}^N|\epsilon_{ij}^{k+1}-\epsilon_{ij}^{k-1}|).
\end{split}
\end{equation}
By \eqref{myeq7.0} we have
\begin{align*}
\epsilon_{ij}^{k-1}&=\langle\xi_j^{k-1},\mathcal F(\Xi^{k-2})\xi_i^{k-1}\rangle,\\
\epsilon_{ij}^{k+1}&=\langle\xi_j^{k+1},\mathcal F(\Xi^{k})\xi_i^{k+1}\rangle.
\end{align*}
Thus by the Hardy inequality and \eqref{myeq7.0.0} there exists a constant $\hat\beta_{\zeta}>0$ independent of $k$ such that
\begin{equation}\label{myeq7.0.2}
|\epsilon_{ij}^{k+1}-\epsilon_{ij}^{k-1}|\leq \hat\beta_{\zeta}(\lVert \Xi^{k}-\Xi^{k-2}\rVert_{\bigoplus_{i=1}^NL^2(\mathbb R^3)}+\lVert\xi_i^{k+1}-\xi_i^{k-1}\rVert+\lVert\xi_j^{k+1}-\xi_j^{k-1}\rVert).
\end{equation}
Since $\Delta$ is $h$-bounded, the result immediately follows from \eqref{myeq7.0.1} and \eqref{myeq7.0.2}.
\end{proof}

\begin{proof}[Proof of the Theorem \ref{main}]
\noindent\textbf{Step 1.} First note that if $\lVert D_{\Phi^{k+1}}-D_{\Phi^{k-1}}\rVert_{2}=0$ for some $k$, then $\mathcal F(\Phi^{k+1})=\mathcal F(\Phi^{k-1})$, and thus $D_{\Phi^{k+2}}=D_{\Phi^{k}}$. Therefore, by induction we have $D_{\Phi^{s}}=D_{\Phi^{s+2}}$ for any $s\geq k$. Then since $\Phi^{k+2t},\ t=0,1,\dots$ (resp., $\Phi^{k+2t+1},\ t=0,1,\dots$) are tuples of the eigenfunctions corresponding to the same eigenvalues of $\mathcal F(\Phi^{k-1})$ (resp., $\mathcal F(\Phi^{k})$), there exist unitary matrices $A_{k+2t}$ (resp., $A_{k+2t+1}$) such that $\lVert A_{k+2t}\Phi^{k+2t}-\Phi^k\rVert=0$ (resp., $\lVert A_{k+2t+1}\Phi^{k+2t+1}-\Phi^{k+1}\rVert=0$) for $t\geq 0$. Hence the results in Theorem \ref{main} are obvious in this case. Therefore, hereafter we assume $\lVert D_{\Phi^{k+1}}-D_{\Phi^{k-1}}\rVert_{2}>0$ for any $k\geq 1$. As in Section \ref{thirdsec}, ${\mathcal E}(\Phi^k,\Phi^{k+1})$ is decreasing with respect to $k$ and converges to some $\mu\in\mathbb R$. If $\mathcal E(\Phi^k,\Phi^{k+1})=\mu$ for some $k$, then we have $\mu=\mathcal E(\Phi^k,\Phi^{k+1})\geq\mathcal E(\Phi^{k+1},\Phi^{k+2})\geq\mu$, and therefore, $\mathcal E(\Phi^k,\Phi^{k+1})=\mathcal E(\Phi^{k+1},\Phi^{k+2})$. Recalling that by Lemma \ref{wpbound} we have
\begin{equation}\label{myeq7.1.1}
{\mathcal E}(\Phi^k,\Phi^{k+1})-{\mathcal E}(\Phi^{k+1},\Phi^{k+2})\geq 2^{-1}\gamma\lVert D_{\Phi^{k+2}}-D_{\Phi^k}\rVert_{2}^2,
\end{equation}
we obtain $\lVert D_{\Phi^{k+2}}-D_{\Phi^{k}}\rVert_{2}=0$, which contradicts the assumption above. Thus we may also assume $\mathcal E(\Phi^k,\Phi^{k+1})>\mu$ for any $k\geq 0$.

We can easily see that for any constants $p\geq q>0$ and $\tilde\theta\in (0,1/2]$ we have
$$p^{\tilde\theta}-q^{\tilde\theta}\geq \frac{\tilde\theta}{p^{1-\tilde\theta}}(p-q).$$
Applying this inequality to $p={\mathcal E}(\Phi^k,\Phi^{k+1})-\mu$, $q={\mathcal E}(\Phi^{k+1},\Phi^{k+2})-\mu$ we obtain
\begin{equation}\label{myeq7.1}
\begin{split}
&({\mathcal E}(\Phi^k,\Phi^{k+1})-\mu)^{\tilde\theta}-({\mathcal E}(\Phi^{k+1},\Phi^{k+2})-\mu)^{\tilde\theta}\\
&\quad\geq \frac{\tilde\theta}{({\mathcal E}(\Phi^k,\Phi^{k+1})-\mu)^{1-\tilde\theta}}({\mathcal E}(\Phi^k,\Phi^{k+1})-{\mathcal E}(\Phi^{k+1},\Phi^{k+2})).
\end{split}
\end{equation}
The factor $({\mathcal E}(\Phi^k,\Phi^{k+1})-{\mathcal E}(\Phi^{k+1},\Phi^{k+2}))$ in the right-hand side is estimated from below by \eqref{myeq7.1.1}.

\noindent\textbf{Step 2.} As for the denominator in the right-hand side of \eqref{myeq7.1}, recalling $\lVert\varphi_i^k\rVert,\lVert\varphi_i^{k+1}\rVert\newline=1$, $1\leq i\leq N$ we can see that
\begin{equation}\label{myeq7.2}
{\mathcal E}(\Phi^k,\Phi^{k+1})=f(\Phi^k,\Phi^{k+1},\mathbf e^k,\mathbf e^{k+1}),
\end{equation}
where $f$ is the functional defined by \eqref{myeq5.0}. 
Set
$$\tilde\Gamma_{\gamma,\mu}:=\{[\Phi,\tilde\Phi,\mathbf e,\tilde{\mathbf e}]: [\Phi,\tilde\Phi]\in \Gamma_{\gamma,\mu},\ \epsilon_i:=\langle\varphi_i,\mathcal F(\tilde\Phi)\varphi_i\rangle,\ \tilde\epsilon_i:=\langle\tilde\varphi_i,\mathcal F(\Phi)\tilde\varphi_i\rangle\},$$
and let $\tilde d$ be the distance function in $(\bigoplus_{i=1}^NH^2(\mathbb R^3))\bigoplus(\bigoplus_{i=1}^NH^2(\mathbb R^3))\bigoplus \mathbb R^N\bigoplus \mathbb R^N$. As in the proof of Lemma \ref{app} using Lemmas \ref{wpbound} and \ref{Ubound} we can show
\begin{equation}\label{myeq7.2.0}
\lim_{k\to\infty}|\langle\varphi_i^k,\mathcal F(\Phi^{k+1})\varphi_i^k\rangle-\langle\varphi_i^k,\mathcal F(\Phi^{k-1})\varphi_i^k\rangle|=0.
\end{equation}
Using Lemma \ref{app}, $\epsilon_i^k=\langle\varphi_i^k,\mathcal F(\Phi^{k-1})\varphi_i^k\rangle$ and \eqref{myeq7.2.0} we can see that
\begin{equation}\label{myeq7.2.1}
\lim_{k\to\infty}\tilde d([\Phi^k,\Phi^{k+1},\mathbf e^k,\mathbf e^{k+1}],\tilde\Gamma_{\gamma,\mu})=0.
\end{equation}

By Lemma \ref{Fredholm} for any $[\Phi',\tilde\Phi',\mathbf e',\tilde{\mathbf e}']\in\tilde \Gamma_{\gamma,\mu}$, $F'(\Phi',\tilde\Phi',\mathbf e',\tilde{\mathbf e}')$ is decomposed into a sum of an isomorphism $L$ and a compact operator $M$. Moreover, extending the domain of $\langle\langle\cdot,\cdot\rangle\rangle$ from $Y_1\times Y_2$ to $Y_2\times Y_2$ in the obvious way, $Y_2$ can be regarded as a real Hilbert space equipped with the inner product $\langle\langle\cdot,\cdot\rangle\rangle$. Since the Fr\'echet derivative is symmetric \cite[Problem 4.3]{Ze}, we have $d^2f(z_1,[z_2,z_3])=d^2f(z_1,[z_3,z_2])=\langle\langle z_2,F'(z_1)z_3\rangle\rangle=\langle\langle z_3,F'(z_1)z_2\rangle\rangle=\langle\langle F'(z_1)z_2,z_3\rangle\rangle$ for any $z_1,z_2,z_3\in Y_1$. Thus $F'(z_1)$ is a symmetric operator with the domain $Y_1\subset Y_2$. The operator $F'(z_1)$ is a sum of $\tilde{\mathcal H}:=\mathrm{diag}\, [h,\dots,h,0,\dots,0]$ ($h$ appears $2N$ times) and a bounded operator, and $\tilde{\mathcal H}$ is a selfadjoint operator with the domain $Y_1$ in the real Hilbert space $Y_2$ equipped with the inner product $\langle\langle\cdot,\cdot\rangle\rangle$, which follows from that $h$ is a selfadjoint operator with the domain $H^2(\mathbb R^3)$ in $L^2(\mathbb R^3)$ equipped with the usual inner product. Thus $F'(z_1)$ is also a selfadjoint operator with the domain $Y_1$. It is also easily seen that $f$ and $F$ are real-analytic. Therefore, we can apply Lemma \ref{Lojasiewicz} to $f$ and see that there exist a neighborhood $U$ of $[\Phi',\tilde\Phi',\mathbf e',\tilde{\mathbf e}']$ and constants $\kappa>0$ and $\theta\in(0,1/2]$ such that
$$|f(\Phi,\tilde\Phi,\mathbf e,\tilde{\mathbf e})-\mu|^{1-\theta}\leq \kappa\lVert F(\Phi,\tilde\Phi,\mathbf e,\tilde{\mathbf e})\rVert_{Y_2},$$
for any $[\Phi,\tilde\Phi,\mathbf e,\tilde{\mathbf e}]\in U$. Because by Lemma \ref{compact} $\Gamma_{\gamma,\mu}$ and therefore, $\tilde\Gamma_{\gamma,\mu}$ are compact, we can choose a finite cover of $\tilde\Gamma_{\gamma,\mu}$ from such neighborhoods. Therefore, by \eqref{myeq7.2.1} there exist $\tilde\kappa>0$, $\tilde \theta\in(0,1/2]$ and $k_1\in\mathbb N$ such that
\begin{equation}\label{myeq7.3}
|f(\Phi^k,\Phi^{k+1},\mathbf e^k,\mathbf e^{k+1})-\mu|^{1-\tilde \theta}\leq \tilde\kappa\lVert F(\Phi^k,\Phi^{k+1},\mathbf e^k,\mathbf e^{k+1})\rVert_{Y_2},
\end{equation}
for any $k\geq k_1$. Since $\lVert\varphi_i^k\rVert, \lVert\varphi_i^{k+1}\rVert=1,\ 1\leq i\leq N$ and $\mathcal F(\Phi^k)\varphi_i^{k+1}=\epsilon_i^{k+1}\varphi_i^{k+1},\ 1\leq i\leq N$, we can see that the $\mathbb R^N\bigoplus \mathbb R^N$ component and the second $\bigoplus_{i=1}^NL^2(\mathbb R^3)$ component of $F(\Phi^k,\Phi^{k+1},\mathbf e^k,\mathbf e^{k+1})$ vanish. Thus we have
$$\lVert F(\Phi^k,\Phi^{k+1},\mathbf e^k,\mathbf e^{k+1})\rVert_{Y_2}=\left(\sum_{i=1}^N\lVert \mathcal F(\Phi^{k+1})\varphi_i^{k}-\epsilon_i^{k}\varphi_i^{k}\rVert_{L^2(\mathbb R^3)}^2\right)^{1/2}.$$
Using $\mathcal F(\Phi^{k-1})\varphi_i^{k}=\epsilon_i^{k}\varphi_i^{k}$ we obtain
\begin{equation}\label{myeq7.4}
\begin{split}
\lVert F(\Phi^k,\Phi^{k+1},\mathbf e^k,\mathbf e^{k+1})\rVert_{Y_2}&=\left(\sum_{i=1}^N\lVert \mathcal F(\Phi^{k+1})\varphi_i^{k}-\mathcal F(\Phi^{k-1})\varphi_i^{k}\rVert_{L^2(\mathbb R^3)}^2\right)^{1/2}\\
&=\left(\sum_{i=1}^N\lVert \mathcal G(\Phi^{k+1})-\mathcal G(\Phi^{k-1}))\varphi_i^{k}\rVert_{L^2(\mathbb R^3)}^2\right)^{1/2}.
\end{split}
\end{equation}

Let us denote by $A_{k+1}^+, A_{k-1}^-$ the $N\times N$ unitary matrices $A$ and $\tilde A$ in Lemma \ref{Ubound} with $\Phi$ and $\tilde \Phi$ replaced by $\Phi^{k+1}$ and $\Phi^{k-1}$ respectively, namely we have
\begin{equation}\label{myeq7.4.1}
\begin{split}
\lVert D_{\Phi^{k+1}}-D_{\Phi^{k-1}}\rVert_2&\geq\lVert A_{k+1}^+\Phi^{k+1}-A_{k-1}^-\Phi^{k-1}\rVert_{\bigoplus_{i=1}^NL^2(\mathbb R^3)}\\
&=\lVert \tilde\Xi^{k+1}_+-\tilde\Xi^{k-1}_-\rVert_{\bigoplus_{i=1}^NL^2(\mathbb R^3)},
\end{split}
\end{equation}
where $\tilde\Xi^{k+1}_+:=A_{k+1}^+\Phi^{k+1}$, $\tilde\Xi^{k-1}_-:=A_{k-1}^-\Phi^{k-1}$. Then it is easily seen that $\mathcal G(\Phi^{k+1})=\mathcal G(\tilde\Xi^{k+1}_+)$, $\mathcal G(\Phi^{k-1})=\mathcal G(\tilde\Xi^{k-1}_-)$. Therefore, recalling Remark \ref{Hboundrem} and using the Hardy inequality and \eqref{myeq7.4.1} we can see that there exists a constant $\breve C$ such that
\begin{align*}
\sum_{i=1}^N\lVert \mathcal G(\Phi^{k+1})-\mathcal G(\Phi^{k-1}))\varphi_i^{k}\rVert_{L^2(\mathbb R^3)}^2&=\sum_{i=1}^N\lVert \mathcal G(\tilde\Xi^{k+1}_+)-\mathcal G(\tilde\Xi^{k-1}_-))\varphi_i^{k}\rVert_{L^2(\mathbb R^3)}^2\\
&\leq \breve C\lVert\tilde\Xi^{k+1}_+-\tilde\Xi^{k-1}_-\rVert_{\bigoplus_{i=1}^NL^2(\mathbb R^3)}^2\\
&\leq \breve C\lVert D_{\Phi^{k+1}}-D_{\Phi^{k-1}}\rVert_{2}^2.
\end{align*}
Combining this inequality, \eqref{myeq7.3} and \eqref{myeq7.4} we obtain
\begin{equation}\label{myeq7.4.2}
|f(\Phi^k,\Phi^{k+1},\mathbf e^k,\mathbf e^{k+1})-\mu|^{1-\tilde\theta}\leq \tilde\kappa\breve C^{1/2}\lVert D_{\Phi^{k+1}}-D_{\Phi^{k-1}}\rVert_{2},
\end{equation}
for $k\geq k_1$.

\noindent\textbf{Step 3.} It follows from \eqref{myeq7.1.1}--\eqref{myeq7.2} and \eqref{myeq7.4.2} that
\begin{equation*}
\begin{split}
&({\mathcal E}(\Phi^k,\Phi^{k+1})-\mu)^{\tilde \theta}-({\mathcal E}(\Phi^{k+1},\Phi^{k+2})-\mu)^{\tilde \theta}\\
&\quad\geq \frac{\tilde \theta}{\tilde\kappa\breve C^{1/2}\lVert D_{\Phi^{k+1}}-D_{\Phi^{k-1}}\rVert_{2}}(2^{-1}\gamma\lVert D_{\Phi^{k+2}}-D_{\Phi^{k}}\rVert_{2}^2),
\end{split}
\end{equation*}
for $k\geq k_1$.
Since the sum of the left-hand side for $k=1,2,\dots$ is finite, the corresponding sum of the right-hand side is also convergent. Setting $\alpha_k:=\lVert D_{\Phi^{k+1}}-D_{\Phi^{k-1}}\rVert_{2}$ this sum is written as
$$\frac{\tilde\theta\gamma}{2\tilde\kappa\breve C^{1/2}}\sum_{k=1}^{\infty}\frac{\alpha_{k+1}^2}{\alpha_k}.$$
Hence by Lemma \ref{series} we can see that
$$\sum_{k=1}^{\infty}\alpha_k=\sum_{k=1}^{\infty}\lVert D_{\Phi^{k+1}}-D_{\Phi^{k-1}}\rVert_{2},$$
is convergent.

Let us define unitary matrices $\tilde A_k$ so that
$$\lVert \tilde A_{k+1}\tilde\Xi_-^{k+1}-\tilde A_{k-1}\tilde\Xi_-^{k-1}\rVert_{\bigoplus_{i=1}^NL^2(\mathbb R^3)}=\lVert\tilde\Xi_+^{k+1}-\tilde\Xi_-^{k-1}\rVert_{\bigoplus_{i=1}^NL^2(\mathbb R^3)},$$
will hold for any $k\geq1$. We set $\tilde A_0:=I, \tilde A_1:=I$, where $I$ is the identity matrix. Assume that $\tilde A_{k-1}$ has been defined. Since $\tilde A_{k-1}$ is unitary and $\Xi_+^{k+1}=A^+_{k+1}(A^-_{k+1})^{-1}\Xi_-^{k+1}$, we have
\begin{align*}
&\lVert\tilde\Xi_+^{k+1}-\tilde\Xi_-^{k-1}\rVert_{\bigoplus_{i=1}^NL^2(\mathbb R^3)}\\
&\quad=\lVert \tilde A_{k-1}A^+_{k+1}(A^-_{k+1})^{-1}\tilde\Xi_-^{k+1}-\tilde A_{k-1}\tilde\Xi_-^{k-1}\rVert_{\bigoplus_{i=1}^NL^2(\mathbb R^3)}.
\end{align*}
Thus we should set $\tilde A_{k+1}:=\tilde A_{k-1}A^+_{k+1}(A^-_{k+1})^{-1}$. Consequently, we obtain 
\begin{align*}
\tilde A_{2k}&=A_2^+(A_2^-)^{-1}\dotsm A_{2(k-1)}^+(A_{2(k-1)}^-)^{-1}A_{2k}^+(A_{2k}^-)^{-1},\\
\tilde A_{2k+1}&=A_3^+(A_3^-)^{-1}\dotsm A_{2(k-1)+1}^+(A_{2(k-1)+1}^-)^{-1}A_{2k+1}^+(A_{2k+1}^-)^{-1},
\end{align*}
for $k\geq 1$. Now set $A_0:=I$, $A_1:=I$ and $A_{2k}:=\tilde A_{2k}A_{2k}^-$, $A_{2k+1}:=\tilde A_{2k+1}A_{2k+1}^-$ for $k\geq 1$. Then if we define $\Xi^k:=A_k\Phi^k$ for $k\geq 0$, we have
\begin{equation}\label{myeq7.4.3}
\lVert \Xi^{k+1}-\Xi^{k-1}\rVert_{\bigoplus_{i=1}^NL^2(\mathbb R^3)}=\lVert \tilde\Xi^{k+1}_+-\tilde\Xi^{k-1}_-\rVert_{\bigoplus_{i=1}^NL^2(\mathbb R^3)},
\end{equation}
for any $k\geq1$. Since $\{\Phi^k\}$ satisfies $\mathcal F(\Phi^k)\varphi_i^{k+1}=\epsilon_i^{k+1}\varphi_i^{k+1}$, $1\leq i\leq N$, we can see that $\Xi^k$ satisfies $\mathcal F(\Xi^k)\xi_i^{k+1}=\sum_{j=1}^N\epsilon_{ij}^{k+1}\xi_j^{k+1}$, $1\leq i\leq N$, where $\epsilon_{ij}^k$ is the $(i,j)$th entry  of the matrix $A_k(\mathrm{diag}\, [\epsilon^k_1,\dots,\epsilon_N^k])A_k^{-1}$. Noting that $A_k$ is a unitary matrix we have $\sum_{i,j=1}^N|\epsilon_{ij}^k|^2=\sum_{i=1}^N|\epsilon_i^k|^2\leq N|\inf\sigma(h)|^2$. Thus we can apply Lemma \ref{reduction} to $\{\Xi^k\}$, which combined with \eqref{myeq7.4.3} and \eqref{myeq7.4.1} yields
\begin{align*}
&\sum_{k=1}^{\infty}\lVert\Xi^{k+1}-\Xi^{k-1}\rVert_{\bigoplus_{i=1}^NH^2(\mathbb R^3)}\\
&\quad\leq2\beta_{\zeta}\sum_{k=1}^{\infty}\lVert\Xi^{k+1}-\Xi^{k-1}\rVert_{\bigoplus_{i=1}^NL^2(\mathbb R^3)} +\lVert\Xi^{2}-\Xi^{0}\rVert_{\bigoplus_{i=1}^NH^2(\mathbb R^3)}\\
&\quad=2\beta_{\zeta}\sum_{k=1}^{\infty}\lVert\tilde\Xi^{k+1}_+-\tilde\Xi_-^{k-1}\rVert_{\bigoplus_{i=1}^NL^2(\mathbb R^3)} +\lVert\Xi^{2}-\Xi^{0}\rVert_{\bigoplus_{i=1}^NH^2(\mathbb R^3)}\\
&\quad\leq 2\beta_{\zeta}\sum_{k=1}^{\infty}\alpha_k+\lVert\Xi^{2}-\Xi^{0}\rVert_{\bigoplus_{i=1}^NH^2(\mathbb R^3)}<\infty,
\end{align*}
with $\zeta:=N^{1/2}|\inf\sigma(h)|$.
Thus there exist limits $\Xi^{\infty}:=\lim_{k\to\infty}\Xi^{2k}$ and $\tilde\Xi^{\infty}:=\lim_{k\to\infty}\Xi^{2k+1}$ in $\bigoplus_{i=1}^NH^2(\mathbb R^3)$. Now noting that $D_{\Xi^k}=D_{\Phi^k}$, that $\lim_{k\to\infty}D_{\Xi^{2k}}=D_{\Xi^{\infty}}$, $\lim_{k\to\infty}D_{\Xi^{2k+1}}=D_{\tilde\Xi^{\infty}}$ with respect to the topology of $\mathcal L(L^2(\mathbb R^3))$, that $D_{\Phi^{2k}}$ and $D_{\Phi^{2k+1}}$ converge in $\mathcal T_2$, and that $\lVert\cdot\rVert_{\mathcal L(L^2(\mathbb R^3))}\leq\lVert\cdot\rVert_{2}$ (cf. \cite[Theorem VI.22 (d)]{RS}) the results in Theorem \ref{main} follow, and the proof is completed.
\end{proof}

\begin{proof}[Proof of Theorem \ref{maincor}]
\noindent (1) Since $\Xi^{2k}$ converges to $\Xi^{\infty}$ in $\bigoplus_{i=1}^NH^2(\mathbb R^3)$, the operator $\mathcal F(\Phi^{2k})-\mathcal F(\Xi^{\infty})=\mathcal G(\Xi^{2k})-\mathcal G(\Xi^{\infty})$ converges to $0$ in $\mathcal L(L^2(\mathbb R^3))$. Thus by the upper semicontinuity of the spectrum (see e.g. \cite[Theorems IV 1.16 and IV 3.18]{Ka}) and the uniform well-posedness, for any $\delta>0$ there exist $k'\in\mathbb N$ and a constant $v\in\mathbb R$ such that the $N$ smallest eigenvalues of $\mathcal F(\Xi^{\infty})$ and $\mathcal F(\Phi^{2k})$ for $k\geq k'$ are smaller than $v$ and the rest of the spectra of them are larger than $v+\gamma-\delta$, which proves (1) for $\mathcal F(\Xi^{\infty})$. The poof for $\mathcal F(\tilde\Xi^{\infty})$ is exactly the same.

\noindent(2) By the proof of (1) there exists a closed curve $g$ in $\mathbb C$ such that the $N$ smallest eigenvalues of $\mathcal F(\Xi^{\infty})$ and $\mathcal F(\Phi^{2k})$ for $k\geq k'$ are enclosed by $g$, and the distances between $g$ and the spectra of $\mathcal F(\Xi^{\infty})$ and $\mathcal F(\Phi^{2k})$ for $k\geq k'$ are larger than $\gamma/3$. Thus using the representation $P_{\Phi}=-(2\pi i)^{-1}\oint_g(\mathcal F(\Phi)-z)^{-1}dz$ of the projections $P_{\Phi}$ to the direct sum of the eigenspaces of $\mathcal F(\Phi)$ we can see that $\lim_{k\to\infty}P_{\Phi^{2k}}=P_{\Xi^{\infty}}$ in $\mathcal L(L^2(\mathbb R^3))$. Hence with $\xi_i^{2k+1}$ in the proof of Theorem \ref{main} we have 
$$\tilde \xi^{\infty}_i=\lim_{k\to\infty}\xi^{2k+1}_i=\lim_{k\to\infty}P_{\Phi^{2k}}\xi_i^{2k+1}=P_{\Xi^{\infty}}\tilde\xi_i^{\infty},\ 1\leq i\leq N,$$
where $\tilde\Xi^{\infty}={}^t(\tilde\xi_1^{\infty},\dots,\tilde\xi_N^{\infty})$. This means that $\tilde\Xi^{\infty}$ is an orthonormal basis of $\mathcal H_{\Xi^{\infty}}:=\mathrm{Ran}\, P_{\Xi^{\infty}}$. In the same way we can also prove that $\Xi^{\infty}$ is an orthonormal basis of the direct sum $\mathcal H_{\tilde\Xi^{\infty}}:=\mathrm{Ran}\, P_{\tilde\Xi^{\infty}}$.

Let $\hat\Phi^{\infty}={}^t(\hat\varphi_1,\dots,\hat\varphi_N)$ be a tuple of the eigenfunctions of $\mathcal F(\Xi^{\infty})$ corresponding to the $N$ smallest eigenvalues $\hat\epsilon_1^{\infty},\dots,\hat\epsilon_N^{\infty}\in\mathbb R$ as above Theorem \ref{maincor}.
Then $\hat\Phi^{\infty}$ is an orthonormal basis of $\mathcal H_{\Xi^{\infty}}$ and
\begin{equation}\label{myeq7.5}
\mathcal F(\Xi^{\infty})\hat\varphi_i^{\infty}=\hat\epsilon_{i}^{\infty}\hat\varphi_i^{\infty},\ 1\leq i\leq N.
\end{equation}
Thus $\hat\Phi^{\infty}$ and $\tilde\Xi^{\infty}$ are orthonormal bases of the same space $\mathcal H_{\Xi^{\infty}}$. Therefore, there exists a unitary matrix $A_{\infty}$ such that $\tilde\Xi^{\infty}=A_{\infty}\hat\Phi^{\infty}$. We note here that $\mathcal F(\tilde\Xi^{\infty})=\mathcal F(\hat\Phi^{\infty})$ also holds.

Next we shall prove $\hat\epsilon_i^{\infty}=\lim_{k\to\infty}\epsilon_i^{2k+1}$, $1\leq i\leq N$. From the proof of Theorem \ref{main} it follows that there exists a Hermitian matrix $(\tilde\epsilon_{ij}^{\infty})$ such that
$$\mathcal F(\Xi^{\infty})\tilde\xi_i^{\infty}=\sum_{j=1}^N\tilde\epsilon_{ij}^{\infty}\tilde\xi_j^{\infty},\ 1\leq i\leq N.$$
Thus by \eqref{myeq7.5} we can see that $\mathrm{diag}\, [\hat\epsilon_1^{\infty},\dots,\hat\epsilon_N^{\infty}]=A_{\infty}^{-1}(\tilde\epsilon_{ij}^{\infty})A_{\infty}$. Since $(\tilde\epsilon_{ij}^{\infty})$ is the limit of the Hermitian matrices $(\epsilon_{ij}^{2k+1})$ whose eigenvalues are $(\epsilon_1^{2k+1},\dots,\epsilon_N^{2k+1})$, the perturbation theorem for the eigenvalues of Hermitian matrices (see e.g. \cite[Problem 1.17]{Ch}) yields $\hat\epsilon_i^{\infty}=\lim_{k\to\infty}\epsilon_i^{2k+1}$, $1\leq i\leq N$.

\noindent(3) If $\Xi^{\infty}=\Theta \tilde\Xi^{\infty}$, we have $\mathcal F(\Xi^{\infty})=\mathcal F(\tilde\Xi^{\infty})$ and thus
$$\mathcal F(\Xi^{\infty})=\mathcal F(\tilde\Xi^{\infty})=\mathcal F(\hat\Phi^{\infty}).$$
Hence by \eqref{myeq7.5} we have
$$\mathcal F(\hat\Phi^{\infty})\hat\varphi_i^{\infty}=\hat\epsilon_{i}^{\infty}\hat\varphi_i^{\infty},\ 1\leq i\leq N,$$
which means that $\hat\Phi^{\infty}$ is a solution to the Hartree-Fock equation.

\noindent(4) Assume that $\hat\Phi^{\infty}$ forms an orthonormal basis of the direct sum of the eigenspaces of the $N$ smallest eigenvalues of $\mathcal F(\hat\Phi^{\infty})$. Then recalling that $\mathcal F(\hat\Phi^{\infty})=\mathcal F(\tilde\Xi^{\infty})$ it follows that $\hat\Phi^{\infty}$ is an orthonormal basis of $\mathcal H_{\tilde\Xi^{\infty}}$. Since $\hat\Phi^{\infty}$ is an orthonormal basis also of $\mathcal H_{\Xi^{\infty}}$, we have $\mathcal H_{\tilde\Xi^{\infty}}=\mathcal H_{\Xi^{\infty}}$, which implies that $\Xi^{\infty}$ and $\tilde \Xi^{\infty}$ are orthonormal bases of the same space.
Therefore, there exists a unitary matrix $\Theta$ such that $\Xi^{\infty}=\Theta\tilde\Xi^{\infty}$.

\noindent(5) This result follows from (2) if we prove that the necessary and sufficient condition that $\Phi,\tilde \Phi\in\mathcal W$ satisfy $D_{\Phi}=D_{\tilde\Phi}$ is that there exists a unitary matrix $\hat A$ such that $\Phi=\hat A\tilde\Phi$. The sufficiency is obvious. The necessity follows from that $\Phi$ is an orthonormal basis of $\mathrm{Ran}\, D_{\Phi}$, which was also mentioned in Section \ref{firstsec}.
\end{proof}

\end{document}